\documentclass{article}
\usepackage{authblk}
\usepackage{tikz-cd}
\usepackage{xcolor}
\usepackage{url}
\usepackage{wrapfig}
\usepackage{float}
\usepackage[margin = 1in]{geometry}
\usepackage{amsmath}
\usepackage{amssymb}
\usepackage{amsfonts}
\usepackage{amsthm}
\usepackage{mathtools}
\usepackage{enumerate}
\usepackage{verbatim}
\usepackage{physics}
\usepackage{euscript}
\usepackage{tabularx}
\usepackage{lineno}
\usepackage{hyperref}
\usepackage[ruled,vlined,linesnumbered]{algorithm2e}
\numberwithin{equation}{section}

\newcommand{\B}{\mathcal{B}}

\newcommand{\OO}{\mathcal{O}}
\newcommand{\RR}{\mathbb{R}}
\newcommand{\CC}{\mathbb{C}}

\newcommand{\QQ}{\mathbb{Q}}
\newcommand{\ZZ}{\mathbb{Z}}
\newcommand{\FF}{\mathbb{F}}

\newcommand{\M}{\mathsf{M}}
\DeclareMathOperator{\End}{End}

\DeclareMathOperator{\Gal}{Gal}

\DeclareMathOperator{\ch}{char}

\DeclareMathOperator{\Hom}{Hom}
\DeclareMathOperator{\Aut}{Aut}

\DeclareMathOperator{\Nrd}{Nrd}
\DeclareMathOperator{\Trd}{Trd}
\DeclareMathOperator{\discrd}{discrd}
\DeclareMathOperator{\disc}{disc}
\DeclareMathOperator{\sgn}{sgn}

\DeclareMathOperator{\llog}{llog}
\newtheorem*{thm*}{Theorem}
\newtheorem{prob}{Problem}

\newtheorem{thm}{Theorem}[section]
\newtheorem{prop}[thm]{Proposition}
\newtheorem{lem}[thm]{Lemma}
\newtheorem{cor}[thm]{Corollary}
\theoremstyle{definition}
\newtheorem{defn}[thm]{Definition}
\newtheorem{rmk}[thm]{Remark}

\newtheorem{heuristic}[thm]{Heuristic}
\newtheorem*{rmk*}{Remark}
\newtheorem*{ex*}{Example}

\DeclareMathOperator{\size}{size}

\DeclarePairedDelimiter\floor{\lfloor}{\rfloor}

\definecolor{jennycolor}{RGB}{37,150,190}

\definecolor{Changningphaabicolor}{RGB}{255,140,0}

 \definecolor{annacolor}{RGB}{37,150,115}

 \definecolor{markcolor}{RGB}{199,21,133}

     \definecolor{todo}{RGB}{255,0,0}

\title{Computing supersingular endomorphism rings using inseparable endomorphisms}

\author[1]{\small Jenny Fuselier}
\author[2,3,4]{\small Annamaria Iezzi}
\author[5]{\small Mark Kozek}
\author[6]{\small Travis Morrison}
\author[7]{\small Changningphaabi Namoijam}

\affil[1]{\footnotesize Department of Mathematical Sciences, High Point University, High Point, NC 27268, USA}
\affil[2]{\footnotesize Univ. Grenoble Alpes, CNRS, Grenoble INP, LJK, 38000 Grenoble, France}
\affil[3]{\footnotesize Dipartimento di Matematica e Applicazioni ``Renato Caccioppoli", Università degli Studi di Napoli Federico II, I-80126 Napoli, Italy }
\affil[4]{\footnotesize Laboratoire GAATI, Université de la Polynésie française, 98702 Faaa, French Polynesia}
\affil[5]{\footnotesize Department of Mathematics \& Computer Science, Whittier College, Whittier, CA 90601, USA}
\affil[6]{\footnotesize Department of Mathematics, Virginia Tech, Blacksburg, VA 24060 USA}
\affil[7]{\footnotesize Department of Mathematics, Colby College, Waterville, ME 04901, USA}

\date{}

\begin{document}

\maketitle

\begin{abstract}
  We give an algorithm for computing an inseparable endomorphism of a supersingular elliptic curve $E$ defined over $\FF_{p^2}$, which, conditional on GRH, runs in expected $O(p^{1/2}(\log p)^2(\log\log p)^3)$  bit operations and requires $O((\log p)^2)$ storage. This matches the  time and storage complexity of the best conditional algorithms for computing a nontrivial supersingular endomorphism, such as those of Eisentr\"{a}ger--Hallgren--Leonardi--Morrison--Park  and Delfs--Galbraith. Unlike these prior algorithms, which require two paths from $E$ to a curve defined over $\FF_p$, the algorithm we introduce only requires one; thus when combined with the algorithm of Corte-Real Santos--Costello--Shi,  our algorithm will be faster in practice. Moreover, our algorithm produces endomorphisms with predictable discriminants, enabling us to prove properties about the orders they generate. With two calls to our algorithm, we can provably compute a Bass suborder of $\End(E)$. This result is then used in an algorithm for computing a basis for $\End(E)$ with the same time complexity, assuming GRH. We also argue that $\End(E)$ can be computed using $O(1)$ calls to our algorithm along with polynomial overhead,  conditional on a heuristic assumption about the distribution of the discriminants of these  endomorphisms. Conditional on GRH and this additional heuristic, this yields a   $O(p^{1/2}(\log p)^2(\log\log p)^3)$ algorithm for computing $\End(E)$ requiring $O((\log p)^2)$ storage.

\end{abstract}

\section{Introduction}

Let $E$ be an elliptic curve defined over a finite field $\FF_q$, where $q$ is a  power of a prime $p$. If $E$ is ordinary, in order to compute the (geometric) endomorphism ring $\End(E)$ of $E$, one must determine the index  $[\End(E):\ZZ[\pi_E]]$ where $\ZZ[\pi_E]$ is the order generated by the Frobenius endomorphism $\pi_E$ of $E$. This problem has been well-studied, and there exist algorithms for computing the endomorphism ring of an ordinary elliptic curve due to Bisson and Sutherland~\cite{BS2011} which run in expected subexponential time, conditional on reasonable heuristics including the Generalized Riemann Hypothesis (GRH). Recently, Robert~\cite{Rob23applications} showed that, given access to a factoring oracle, there is a polynomial-time algorithm for computing the endomorphism ring of an ordinary elliptic curve. 

When $E$ is supersingular, its endomorphism algebra $\End^0(E)\coloneqq \End(E)\otimes \QQ$ is a quaternion algebra, and $\End(E)$ is a maximal order of $\End^0(E)$. In this case, there is no canonical imaginary quadratic order which embeds in $\End(E)$. Even worse, if we have a suborder $\Lambda\subseteq\End(E)$, there can be exponentially (in $\log(\disc(\Lambda))$) many pairwise non-isomorphic maximal orders  which contain $\Lambda$. This stands in contrast to the ordinary case where we have a canonical embedding of a finite-index suborder and there is a unique maximal order containing both this suborder and $\End(E)$: this maximal order is the ring of integers of the imaginary quadratic number field $\QQ(\pi_E)\cong  \End^0(E)$. %

This suggests that computing the endomorphism ring of a supersingular elliptic curve is a hard problem. More precisely, there are no known efficient algorithms for solving the {\em endomorphism ring} problem: 
\begin{prob}\label{prob:EndE}
    Given  a supersingular elliptic curve $E$ defined over $\FF_{p^2}$,  compute the endomorphism ring $\End(E)$ of $E$, that is, compute a basis of the maximal order $\mathcal O$ in the quaternion algebra $B_{p,\infty}$ such that $\End(E)\cong  \mathcal O$. 
\end{prob}
The assumption that Problem~\ref{prob:EndE}  is hard is central to the security of isogeny-based cryptography. Indeed, in isogeny-based cryptosystems, a secret key is an isogeny of large, smooth degree between two supersingular elliptic curves -- a path in the supersingular isogeny graph -- and the problem of path-finding in supersingular isogeny graphs has been proven to be equivalent to the problem of computing supersingular endomorphism rings, assuming GRH (see Eisentr\"{a}ger, Hallgren, Lauter, Morrison, and Petit~\cite{EHLMP}, Wesolowski \cite{Wes22} and Page and Wesolowski \cite{PW24}).

There are two approaches to computing the endomorphism ring of a supersingular elliptic curve $E$. One approach uses the reduction of~\cite{EHLMP} to path-finding in isogeny graphs, and the second, perhaps more straightforward approach, involves computing endomorphisms of $E$ until computing a generating set for $\End(E)$. In the first approach, one computes a supersingular curve $E_0$ with known or easily computable endomorphism ring, which can be done efficiently assuming GRH with Br\"oker's algorithm~\cite{Brok09}, and then computes an isogeny $E_0\to E$. With the isogeny $E_0\to E$ and $\End(E_0)$, one can efficiently compute $\End(E)$ via the reduction in~\cite{EHLMP}. However,  in isogeny-based cryptosystems such as SQIsign~\cite{sqisign},  there is a curve $E_0$ with known endomorphism ring as a public parameter for the cryptosystem, a user's public key is another supersingular elliptic curve $E$, and their corresponding private key is a secret isogeny $E_0\to E$. Thus  the reduction to path-finding gives a roundabout attack: any  isogeny $E_0\to E$ is (at least functionally) the secret key that an attacker wishes to compute, so an attacker   would not need $\End(E)$ after having computed any isogeny $E_0\to E$. This motivates an investigation into the second approach to computing $\End(E)$ and hence into the design of algorithms for computing a single endomorphism of $E$. This paper focuses on the design of such an algorithm and an investigation into what can be proved about an order generated by the endomorphisms output by a few calls to that algorithm. 

 The first algorithm for computing nontrivial endomorphisms of a supersingular elliptic curve $E$ is due to Kohel~\cite{Koh96} and runs in time 
$O(p^{1+\epsilon})$ for any $\epsilon>0$; in Kohel's strategy one first computes a spanning tree of the $\ell$-isogeny graph rooted at $E$ and then adds two edges to produce two cycles and thus two endomorphisms of $E$. These two cycles generate a suborder. Delfs and Galbraith~\cite{DG16} compute an endomorphism of $E$ by finding  two distinct isogenies $\psi_i\colon E\to E_i$ to two distinct $\FF_p$-rational curves $E_i$, $i=1,2$, solve the easier problem of path-finding in the $\FF_p$-rational isogeny graph of $\FF_p$-rational supersingular elliptic curves to compute an isogeny $\phi\colon E_1\to E_2$, and return the endomorphism $\widehat{\psi_2}\circ \phi\circ \psi_1$. The complexity of finding an isogeny $\psi_i\colon E\to E_i$ with $E_i$ defined over $\FF_p$ is $\widetilde{O}(p^{1/2}),$ conditional on GRH, while the complexity of the algorithm of~\cite{DG16} for  pathfinding in the $\FF_p$-subgraph is $\widetilde{O}(p^{1/4})$.  Eisentr\"{a}ger, Hallgren, Leonardi, Morrison, and Park~\cite{EHLMP20} give an algorithm for computing a cycle, based at $E$, in the $\ell$-isogeny graph $G(p,\ell)$ by finding two distinct isogenies $\phi_i\colon E\to E^{(p)}$, $i=1,2$, where $E^{(p)}$ is the codomain of the $p$-power Frobenius isogeny $\pi\colon E\to E^{(p)}$. Then $\widehat{\phi_2}\circ\phi_1$ is an endomorphism of $E$.  An isogeny $E\to E^{(p)}$ is computed by first using random walks in the $\ell$-isogeny graph to find an isogeny $\psi_1\colon E\to E_1$ where $E_1$ is defined over $\FF_p$; then $\widehat{\psi_1^{(p)}}\circ\psi_1$ is an isogeny $E\to E^{(p)}$, where $\psi^{(p)}$ is the isogeny obtained by the action of the Frobenius automorphism on $\psi$ (see Section~\ref{subsec:Galoisinvolution}). The latter two algorithms for computing a nontrivial endomorphism have the same asymptotic complexity, since both require two paths from $E$ to the $\FF_p$-subgraph, but the algorithm of~\cite{EHLMP20} is strictly faster since the overhead is polynomial in $\log p$ compared to the exponential overhead $\widetilde{O}(p^{1/4})$ required to find a path in the $\FF_p$-subgraph in the algorithm of~\cite{DG16}.

In this paper we introduce an algorithm, Algorithm~\ref{alg:insependo}, for computing certain inseparable endomorphisms of $E$ which we define as {\em inseparable reflections} in Section~\ref{sec:inseparables}. The idea is simple: to compute an inseparable endomorphism of $E$, compute an isogeny $\psi\colon E\to E^{(p)}$ as described above, and return $\pi\circ\psi$ where $\pi$ is the $p$-power Frobenius.  Assuming the Generalized Riemann Hypothesis, it terminates in expected $\widetilde{O}(p^{1/2})$ bit operations, it is low storage and easy to parallelize (unlike a generic low-storage collision algorithm such as Pollard's $\rho$). Algorithm~\ref{alg:insependo} is twice as fast as the algorithms in~\cite{DG16, EHLMP20} since it requires only one isogeny to a $\FF_p$-rational curve, rather than 2. Thus Algorithm~\ref{alg:insependo} reflects the state-of-the-art in conditional algorithms for computing a nontrivial endomorphism of a supersingular elliptic curve (when combined with the algorithm of Santos--Costello--Shi~\cite{CCS22} for fast subfield detection). One might suspect that the fact that the output is always an inseparable endomorphism of the input curve $E$ might be an obstacle if one was trying to use this algorithm to compute a generating set for $\End(E)$. To the contrary, we show that Algorithm~\ref{alg:insependo} improves on all previous algorithms for computing endomorphisms in a second way: we are able to control the arithmetic properties of orders generated by endomorphisms output by Algorithm~\ref{alg:insependo}. More precisely, Algorithm~\ref{alg:insependo} has auxiliary inputs $\ell$ (a prime) and $d$ (a positive square-free integer) that can be chosen so that  the output endomorphism generates an imaginary quadratic order that is maximal at every prime except at $\ell$. As a consequence, with an appropriate choice of the auxiliary inputs, two endomorphisms output by the algorithm will generate a Gorenstein order (Proposition~\ref{prop:gorenstein}) and, with slightly more care, a Bass order (\ref{thm:bass_from_inseps}). 

We give two algorithms for computing $\End(E)$ using Algorithm~\ref{alg:insependo}: a rigorous (assuming GRH) algorithm and a simple but heuristic algorithm. Let us first outline the rigorous version of the algorithm. First,  Algorithm~\ref{alg:bass} calls Algorithm~\ref{alg:insependo} twice to produce a basis for a Bass suborder of $\End(E)$:
\begin{thm*}[Theorem~\ref{thm:makebass}]
    On input  a supersingular elliptic curve $E$ defined over $\FF_{p^2}$, Algorithm~\ref{alg:bass} computes a basis of a Bass suborder of $\End(E)$. Assuming the Generalized Riemann Hypothesis, the algorithm terminates in an expected $O(p^{1/2}(\log p)^2(\log\log p)^3)$ number of bit operations. 
\end{thm*}
From a theoretical viewpoint, it suffices to compute a Bass suborder $\OO$ of $\End(E)$: by~\cite[Proposition 5.2]{EHLMP20}, the number of maximal overorders containing a given Bass order $\Lambda$ is bounded by a quantity growing subexponentially in the size of $\Lambda$\footnote{Actually, this is proven in ~\cite{EHLMP20} under the additional assumption that $\Lambda$ is hereditary (i.e. its reduced discriminant is square-free),  but this assumption is not necessary; see the proof of Theorem~\ref{thm:bass_from_inseps}.}, and one can efficiently enumerate these maximal overorders.  Using algorithms from~\cite{KLPT, EHLMP, Wes22}, one can efficiently decide whether a given maximal order $\OO\supseteq\Lambda$ is isomorphic to $\End(E)$. Building on ideas in~\cite{EHLMP20}, Algorithm~\ref{alg:endo_ring} computes $\End(E)$  by computing a Bass suborder $\Lambda$ of $\End(E)$, and then enumerates maximal orders $\OO\supseteq \Lambda$ until finding $\OO\cong \End(E)$. In Section~\ref{sec:algorithms}, we show that one can remove the heuristic assumptions needed in~\cite{EHLMP20} and we prove the following theorem:

\begin{thm*}[Theorem~\ref{thm:EndE}]
    There exists an algorithm (Algorithm~\ref{alg:endo_ring}) which takes as input a supersingular elliptic curve $E$ defined over $\FF_{p^2}$ and returns a basis of a maximal order $\OO$ contained in the quaternion algebra ramified at $\{p,\infty\}$ such that $\End(E)\cong  \OO$. Assuming the Generalized Riemann Hypothesis, Algorithm~\ref{alg:endo_ring} terminates in an expected $O(p^{1/2}(\log p)^2(\log\log p)^3)$ number of bit operations. 
\end{thm*}

 Summarizing the above theorems, two endomorphisms produced by Algorithm~\ref{alg:insependo} generate a Bass suborder of $\End(E)$, and with subexponential overhead, we compute a basis for $\End(E)$. 
 
 A more straightforward approach to computing $\End(E)$ is to compute several endomorphisms until finding a generating set. Suppose we have an algorithm which generates a random endomorphism of a supersingular elliptic curve $E$ defined over $\FF_{p^2}$. What is the expected number of  calls to that algorithm before  finding a set of endomorphisms which generate $\End(E)$ as an order? In~\cite{GPS}, Galbraith, Petit, and Silva give a heuristic argument that this expectation is $O(\log p)$. Our work in Section~\ref{sec:heuristic} suggests that this estimate is pessimistic: the expected number of calls is bounded by a constant, assuming a reasonable heuristic on the distribution of the discriminants of random endomorphisms. 
 
 In this paper, we focus on computing $\End(E)$ with  endomorphisms output by Algorithm~\ref{alg:insependo}.
We first remark that no collection of inseparable endomorphisms can generate $\End(E)$: the endomorphisms produced by Algorithm~\ref{alg:insependo} are inseparable and hence belong to $P$,  the $2$-sided ideal of inseparable endomorphisms of $E$. We show in Proposition~\ref{prop:z+p} that $\ZZ+P$ is the unique suborder of $\End(E)$ of index $p$, and the only maximal order containing $\ZZ+P$ is $\End(E)$. In Section~\ref{sec:heuristic}, we show that the expected number of calls to Algorithm~\ref{alg:insependo} before finding a generating set for $\ZZ+P$ is bounded by a positive constant that is not dependent on either $p$ or $E$, assuming Heuristic~\ref{heuristic:coprime} which concerns the distribution of discriminants of endomorphisms produced by Algorithm~\ref{alg:insependo}.  Finally, with a basis of  $\ZZ+P$, one can efficiently compute a basis of $\End(E)$ using algorithms due to Voight~\cite{V2013}. While this approach is no faster than the enumeration-style approach in Algorithm~\ref{alg:endo_ring}, it is simpler to implement: it requires only an implementation of Algorithm~\ref{alg:insependo}, an implementation of a generalization~\cite{BCEMP} of Schoof's algorithm~\cite{SchoofSEA}, and linear algebra to compute a basis for an order in the quaternion algebra isomorphic to $\End(E)$. Our implementation of this algorithm and all necessary subroutines in SageMath~\cite{sagemath} is available at~\url{https://github.com/travismo/inseparables}.

In conclusion, we prove that two calls to Algorithm~\ref{alg:insependo} produce a Bass order unconditionally, and, assuming  Heuristic~\ref{heuristic:coprime}, only $O(1)$ calls to Algorithm~\ref{alg:insependo} (along with subexponential overhead) produce $\End(E)$. Alternatively, by Theorem 1.1 of~\cite{ES24}, two calls to Algorithm~\ref{alg:insependo}, the factorization of the discriminant of the order generated by the output, and polynomial overhead will produce $\End(E)$. 

Note that by Theorem 8.8 of~\cite{PW24}, there is an algorithm for computing $\End(E)$ in $\widetilde{O}(p^{1/2})$ bit operations unconditionally, but the algorithm also requires $\widetilde{O}(p^{1/2})$ storage. On the other hand, the heuristic algorithm for computing $\End(E)$ we provide in Section~\ref{sec:heuristic} -- compute inseparable endomorphisms with Algorithm~\ref{alg:insependo} until finding a basis for $\ZZ+P$ and then recover a basis for $\End(E)$ with linear algebra --  requires only $\text{polylog}(p)$ storage and terminates in expected $O(p^{1/2}(\log p)^2(\log\log p)^3)$ bit operations. Our  provable variant has the same asymptotic time complexity, and the storage complexity is determined by the storage used to factor a single integer of magnitude $O(p^4)$. 

The paper is organized as follows. In Section \ref{sec:background}, we review the mathematical background of the paper and fix our notation.  In Section \ref{sec:inseparables}, we study the properties of the suborder $\mathbb Z+P\subseteq \End(E)$, where $P$ is the ideal of inseparable endomorphisms of $E$. We also define inseparable reflections, building on the definition of $(d,\epsilon)$-structures of Chenu and Smith~\cite{CS21}, and study the structure of quaternionic orders generated by inseparable reflections. In particular we determine when they generate Gorenstein (Proposition~\ref{prop:gorenstein}) and Bass (Theorem~\ref{thm:bass_from_inseps}) orders in $\End(E)$. Section \ref{sec:algorithms}  makes the ideas in Section~\ref{sec:inseparables} effective. First, we analyze Algorithm~\ref{alg:insependo}, which computes inseparable endomorphisms of $E$. Next, we use Algorithm~\ref{alg:insependo} in Algorithm~\ref{alg:bass} to compute a Bass suborder of $\End(E)$. In Section~\ref{sec:provable}, we introduce Algorithm~\ref{alg:endo_ring}, which provably computes $\End(E)$ and, conditional on GRH but no further heuristics, terminates in $\widetilde{O}(p^{1/2})$ time. Algorithm~\ref{alg:endo_ring} calls  Algorithm~\ref{alg:bass}, along with  algorithms of~\cite{EHLMP,EHLMP20,Wes22}, to compute a basis for $\End(E)$. Finally, in Section \ref{sec:inseparablegens}, we propose a heuristic algorithm to compute $\End(E)$ in which we first find enough inseparable reflections to generate $\ZZ+P$ and then use linear algebra to compute a basis for $\End(E)$ from a basis for $\ZZ+P$. In  Appendix~\ref{sec:appendix}, we discuss some of the algorithmic aspects of this approach. 

\subsection*{Acknowledgements}

We thank Heidi Goodson, Christelle Vincent, and Mckenzie West for organizing the first edition of  \emph{Rethinking Number Theory} in 2020, where this project began, and  the American Institute of Mathematics for their additional support. We also thank John Voight for several helpful discussions. We thank the referee for several useful comments and suggestions that improved the presentation of the results.

The second author was partially supported by the European Union - FSE-REACT-EU, PON Research and Innovation 2014-2020 DM1062/2021 contract number 18-I-15358-2, by the Italian National Group for Algebraic and Geometric Structures and their Applications (GNSAGA - INdAM), by Projet ANR collaboratif Barracuda ANR-21-CE39-0009-BARRACUDA, and by the ANR Grant ANR-20-CE40-0013-MELODIA. The fourth author was partially supported by the National Science Foundation grant CNS-2340564 and the Commonwealth Cyber Initiative. The fifth author was partially supported by MOST Grant 110-2811-M-007-517.

 \section{Notation and background}\label{sec:background}
 In this section we fix our notation and recall some  definitions and facts about elliptic curves and quaternion algebras. We refer the reader to Silverman \cite[Chapters III and V]{AEC} and Voight \cite{Voight} for details. 

\subsection{Elliptic curves}
Let $q$ be a positive power of a prime $p>3$, and let $E$ be an elliptic curve defined over the finite field $\FF_q$. Since isomorphic elliptic curves have isomorphic endomorphism rings, we may always assume that $E$ is defined  by a short Weierstrass affine form $E:y^2=x^3+ax+b$, with $a,b\in \mathbb F_{q}$,  such that $4a^3+27b^2\neq 0$. An {\em isogeny} $\phi\colon E\to E'$ between two elliptic curves is a non-constant rational map inducing a group homomorphism $E(\overline{\FF_q})\to E'(\overline{\FF_q})$. An {\em endomorphism} of $E$ is either an isogeny $E\to E$ or the zero-map on $E$.   We define the elliptic curve $E^{(p)}: y^2=x^3+a^px+b^p$, and we denote by $\pi$  the $p$-power Frobenius isogeny $\pi\colon E \to E^{(p)}$ defined by $\pi(x,y)=(x^p,y^p)$. We use the same notation $\pi$ for every such Frobenius isogeny, independent of the choice of the starting elliptic curve. We let $\pi_E$ denote the Frobenius endomorphism which sends $(x,y)\mapsto (x^q,y^q)$. For an integer $n$, we  denote by $E[n]$ the {\em $n$-torsion subgroup} of $E$, consisting of points of $E$ of order dividing $n$. The elliptic curve $E$ is {\em supersingular} if and only if $E[p]=\{0\}$.

For elliptic curves $E,E'$ defined over $\FF_q$, we use the notation $\Hom(E,E')$ for the set of isogenies from $E$ to $E'$ defined over $\FF_q$ together with the zero map. If $L/\FF_{q}$ is an algebraic extension, we let $E_L$ denote the base change of $E$ from $\FF_q$ to $L$ and let $\Hom_L(E,E')\coloneqq \Hom(E_L,E_L')$. Finally we call $\End(E)\coloneqq \Hom_{\overline{\FF_q}}(E,E)$ the  (geometric) \emph{endomorphism ring of $E$} and $\End^0(E)\coloneqq \End(E)\otimes_{\ZZ} \QQ$ the (geometric) \emph{endomorphism algebra of $E$}. When $E$ is a supersingular elliptic curve defined over $\FF_q$, $E$ has a model defined over $\FF_{p^2}$ since its $j$-invariant is in $\FF_{p^2}$. Moreover, we can choose a model of $E$ so that all of its isogenies are defined over $\FF_{p^2}$ as well: indeed we can choose a model so that the trace of $\pi_E$ of $E$ is $2p$, in which case $\pi_E=[p]$, the multiplication-by-$p$ map. If $\psi\colon E\to E'$ is an isogeny between any two such models of  elliptic curves $E$ and $E'$, then $\psi\pi_{E}=\pi_{E'}\psi$ and so $\psi$ is defined over $\FF_{p^2}$ as desired.

In this paper, we focus on supersingular elliptic curves over $\FF_{p^2}$, although some of the results are stated for  elliptic curves over $\mathbb F_q$.  If $E/\mathbb F_{p^2}$ is a supersingular elliptic curve, then $\End^0(E)$ is isomorphic to  the definite quaternion algebra $B_{p,\infty}$ over $\mathbb Q$ ramified exactly at $p$ and $\infty$, and $\End(E)$ is  a maximal order in $\End^0(E)$. 
Computing $\End(E)$ entails finding a basis of a maximal order $\mathcal O$ in $B_{p,\infty}$ such that $\End(E)\simeq  \mathcal O$.

\subsubsection{Isogeny graphs}

We now define supersingular isogeny graphs; see~\cite{Mes86},~\cite[Chapter 7]{Koh96}, and~\cite[Section 3]{secuer} for additional details. Let $p>3$ and $\ell$ be distinct primes. The {\em supersingular $\ell$-isogeny graph in characteristic $p$}, denoted by $G(p,\ell)$, is a directed multigraph, consisting of supersingular elliptic curves over $\overline{\FF_p}$ and their $\ell$-isogenies. 
More precisely, the vertex set of $G(p,\ell)$ is $V=V(p)$, a complete set of 
representatives of isomorphism classes of supersingular elliptic curves over $\FF_{p^2}$. For $E,E'\in V$, the arrows in $G(p,\ell)$ from $E$ to $E'$ are a complete set of representatives of equivalence classes of $\ell$-isogenies $E\to E'$, where two $\ell$-isogenies $\phi,\psi\colon E\to E'$ are equivalent if $\phi=u\psi$ for some automorphism $u\in \Aut(E)$.   This graph is finite, with approximately $(p-1)/12$ many vertices: indeed, by~\cite[V.4.1(c)]{AEC}, the number of vertices in $G(p,\ell)$ is 
\[
\#V(p) = \floor*{\frac{p-1}{12}} + \begin{cases} 0 &: p\equiv 1 \pmod{12} \\ 1 &: p\equiv 5,7\pmod{12} \\ 2 &: p\equiv 11\pmod{12} \end{cases}
\]
 and the  out-degree at each vertex is constant, equal to $\ell+1$. 

Consider the $\CC$-vector space $H$ with basis $V$. Let $a_{E\to E'}$ denote the number of cyclic subgroups $C\leq E[\ell]$ of order $\ell$ such that $E/C\simeq E'$. Define the adjacency operator $A=A(\ell,p)$ on $H$ by $AE = \sum_{E'} a_{E\to E'}E'$. Define $w_E\coloneqq \#\Aut(E)/2$. The vector space $H$ is equipped with an inner product defined by $\langle E,E'\rangle = w_E$ if $E=E'$ and $0$ otherwise. Define $\mathcal{E} = \sum_{E\in V} w_E^{-1} E$. Then 
\[
\langle \mathcal{E}, \mathcal{E}\rangle = \sum_{E\in V} w_E^{-1} = \frac{p-1}{12}
\]
and $\mathcal{E}$ is an eigenvector for $A$ with eigenvalue $\ell+1$. The adjacency operator $A$ is self-adjoint as an operator on $H$ with respect to $\langle\cdot,\cdot \rangle$. Thus $A$ has all real eigenvalues. 
Moreover, $G(p,\ell)$ is a (directed) {\em Ramanujan graph}: the magnitude of the second largest eigenvalue of $A$ is bounded by $2\sqrt{\ell}$. 

This implies that the  random walk in $G(p,\ell)$ mixes rapidly, a fact that we exploit in our algorithms for computing endomorphisms of supersingular elliptic curves. More precisely, a {\em probability distribution} on $G(p,\ell)$ is a vector $v=\sum_{E\in V} v_E E\in H$ such that $\sum_{E\in V} v_E=1$ and $v_E\geq 0$ for all $E\in V$. The random walk on $G(p,\ell)$ is the Markov chain defined by the transition matrix $P\coloneqq \frac{1}{\ell+1}A$. Then $s=\mathcal{E}/\langle \mathcal{E}, \mathcal{E}\rangle$ is the stationary distribution for the random walk. Let $v$ be any probability distribution on $G(p,\ell)$ and $v^{(t)}\coloneqq  P^tv$, the probability distribution obtained by sampling  according to $v$ and then taking $t$ many random steps in $G(p,\ell)$. The Ramanujan property guarantees that as $t\to \infty$, the sequence of distributions $v^{(t)}$ rapidly converges to $s$: for example, we have that the total variation distance between $v^{(t)}$ and $s$ is $O(p^{-1/2})$ if $t=\Omega(\log p)$, where the implied constants depend on $\ell$ but not on $p$ or $t$. 

We can compute random walks in $G(p,\ell)$ using the $\ell$th classical modular polynomial $\Phi_{\ell}(X,Y)\in \ZZ[X,Y]$, which over a field $k$ with $\ch(k)\not=\ell$ parameterizes  $\overline{k}$-isomorphism classes of elliptic curves connected by an $\ell$-isogeny with cyclic kernel. Given a supersingular elliptic curve, by selecting a random root $j$ of $\Phi_{\ell}(j(E),Y)$ (weighted according to its multiplicity as a root), we can effectively take a random step from $E$ to one of its neighbors in $G(p,\ell)$.
\subsection{Quaternion algebras}\label{background_quaternion_algebras}
Let $F$ be a field. A quaternion algebra $B$ over $F$ is a central simple $F$-algebra of dimension
$4$. Let $a,b\in F^\times$, and let $H(a,b) \coloneqq F \oplus F i\oplus Fj
\oplus F ij$ be the $F$-algebra with $F$-basis $\{1,i,j,ij\}$ subject to the multiplication rules
$i^2=a$, $j^2=b$, and $ij=-ji$. Then, $H(a,b)$ is a quaternion algebra. Moreover, assuming that the characteristic of $F$ is not 2, for any
quaternion algebra $B$ over $F$, there exist $a,b\in F$ such that $B$ is isomorphic to  $H(a,b)$.

\subsubsection{The canonical involution}
Let $B=H(a,b)$ be a quaternion algebra over $F$ with basis $\{1,i,j,ij\}$.
The {\em standard involution} of $B$ is the $F$-linear map $\,\bar{}\colon B\to B$ such that if $\alpha=w+xi+yj+zij\in B$, then $\bar{\alpha}=w-xi-yj-zij$. Note that it satisfies $\bar{1} = 1$, $\bar{\bar{\alpha}} = \alpha$, and $\overline{\alpha\beta} = \bar{\beta}\bar{\alpha}$ for every $\alpha, \beta\in B$.
We define the {\em reduced trace} of $\alpha\in B$ to be $\Trd\alpha\coloneqq \alpha+\bar{\alpha}$ and  the {\em reduced norm of $\alpha$} to be $\Nrd\alpha\coloneqq\alpha\bar{\alpha}$. Both $\Trd\alpha$ and $\Nrd\alpha$ are in $F$ for any $\alpha \in B$. Note that $\alpha$ and $\bar{\alpha}$ are roots of their characteristic polynomial $x^2-(\Trd\alpha)x+\Nrd\alpha$.

The reduced trace defines a pairing 
$\langle\cdot,\cdot\rangle\colon B\times B\rightarrow F$ 
defined by  $(\alpha,\beta) \mapsto \Trd(\alpha\bar{\beta})$.
The corresponding quadratic form $Q:B\rightarrow F$ is defined by $Q(\alpha)=\Nrd(\alpha),$ for $\alpha \in B$. Now, let $\mathcal B=\{e_1,e_2,e_3,e_4\}$ be a basis of $B$. We define the \emph{Gram matrix of $Q$} with respect to the basis $\mathcal B$ as the matrix
$$G=\left(\langle e_i, e_j\rangle\right)_{1\leq i,j\leq 4}=\left(\Trd(e_i\bar{ e_j})\right)_{1\leq i,j\leq 4}.$$
Then, for $\alpha=x_1e_1+x_2e_2+x_3e_3+x_4e_4$ and $\beta=y_1e_1+y_2e_2+y_3e_3+y_4e_4$, with $x_i,y_i\in F$, we have
\[
\langle \alpha,\beta\rangle=\Trd(\alpha\bar{\beta})=xGy^t,
\]
where $x=(x_1,x_2,x_3,x_4)$ and $y=(y_1,y_2,y_3,y_4)$.
\subsubsection{Completions, splitting, and ramification}

Let $\QQ_v$ denote the completion at a place $v$ of $\QQ$. Here, $\QQ_v=\QQ_p$ for some prime $p$ if $v$ is a finite place, and $\QQ_v=\RR$ if $v$ is the infinite place. 
If $B$ is a quaternion algebra over $\QQ$, then
$B\otimes \QQ_v$ is a quaternion algebra over $\QQ_v$. A quaternion algebra over $\QQ_v$ is
either the unique division algebra of dimension 4 over $\QQ_v$ or is isomorphic to $M_2(\QQ_v)$. If
$B\otimes \QQ_v \simeq M_2(\QQ_v)$, we say that $B$ is {\em split
at $v$}. If $B\otimes \QQ_v$ is a division algebra, we say that $B$ is {\em ramified at $v$}. The set of places of $\QQ$ where $B$ is ramified is a finite set of even cardinality. If $B$ is not ramified at any place, then $B\simeq M_2(\QQ)$. The \emph{discriminant $\disc(B)$ of $B$} is the product of all primes $p$ at which $B$ is ramified.
\subsubsection{Quaternionic ideals and orders}
Let $B$ be a quaternion algebra over $\QQ$. A $\ZZ$-{\em lattice} $I$ in $B$ is a finitely generated $\ZZ$-submodule of $B$ such that $\QQ I = B$. A {\em $\ZZ$-order} $\OO\subseteq B$ is a $\ZZ$-lattice in $B$ which is also a subring. 
Analogously, one defines a $\ZZ_p$-order in the quaternion algebra  $B\otimes \QQ_p$. Given a lattice $I$ in $B$, the {\em left order} of $I$ is $\OO_L(I)\coloneqq \{\alpha\in B: \alpha I\subseteq I\}$, and we similarly define its right order $\OO_R(I)\coloneqq \{\alpha \in B: I\alpha \subseteq I\}$. A lattice $I\subseteq B$ is a {\em left} (resp. {\em right}) {\em fractional $\OO$-ideal} if $\OO\subseteq \OO_L(I)$  (resp. $\OO\subseteq \OO_R(I)$), and a  fractional left $\OO$-ideal $I$ is an {\em integral left $\OO$-ideal} (or simply a left ideal  of $\OO$) if $I\subseteq \OO$.  If $I$ is both a left and right $\OO$-ideal, we say that $I$ is a {\em two-sided} ideal of $\OO$. For a left  (or right) $\OO$-ideal $I$, define the {\em reduced norm}  of $I$ to be $\Nrd(I)\coloneqq \gcd(\{\Nrd(\alpha) : \alpha \in I\})$. 

An order $\OO \subseteq B$
is {\em maximal} if it is not properly contained in any other order. There
can exist distinct maximal orders in $B$ which can even be non-isomorphic.

The situation is a little simpler for $B\otimes \QQ_p$. Indeed, if $B$ is split at $p$, 
there are infinitely many maximal orders in $B\otimes \QQ_p$, but they are all conjugate to $M_2(\ZZ_p)$. 
If
$B\otimes \QQ_p$ is a division algebra, then 
one can extend the valuation on $\QQ_p$ to $B\otimes \QQ_p$, and the unique maximal order
is the valuation ring.
A $\ZZ$-order $\OO\subseteq B$ is maximal if and only if $\OO\otimes \ZZ_p$ is a maximal $\ZZ_p$-order in $B\otimes \QQ_p$ for every prime $p$~\cite[Lemma 10.4.3]{Voight}. Thus, maximality of an order in $B$ is a local property. 

We can define the notion of discriminant also for an order $\OO \subseteq  B$. Let $\alpha_1, \alpha_2,\alpha_3,\alpha_4$ be a $\ZZ$-basis of $\OO$, then the {\em discriminant $\disc(\OO)$} is defined as %
\[
\disc(\OO):= \det(\langle \alpha_i,\alpha_j\rangle)_{1\leq i,j\leq 4}=\det(\Trd(\alpha_i\bar{\alpha_j})))_{1\leq i,j\leq 4}\in \ZZ.
\]
It is possible to show that 
$\disc(\OO)$ is always a square, so we define the \emph{reduced discriminant} $\discrd(\OO)$ of $\OO$ to be the positive integer satisfying $\discrd(\OO)^2 =\disc(\OO)$. 
A $\ZZ$-order $\mathcal O$ is maximal in $B$ if and only if $\discrd(\OO)=\disc(B)$ \cite[Theorem 15.5.5]{Voight}. Moreover, if $\OO\subseteq \OO'$,  then $\discrd(\OO)=[\OO'\colon \OO]\discrd(\OO')$, where $[\OO'\colon \OO]$ denotes the index of $\OO$ in $\OO'$ as abelian groups \cite[Lemma 15.2.15]{Voight}.

We recall some of the properties of orders in a quaternion algebra $B$ over $\QQ$.  We say that a $\ZZ$-order $\OO \subseteq B$ is \emph{Gorenstein} if 
every left ideal $I$ of $\OO$ satisfying $\OO_L(I)=\OO$ is invertible. The order $\OO$ is \emph{Bass} if every overorder $\OO'\supseteq \OO$ is Gorenstein. An order $\OO$ is Bass if and only if it is {\em basic}, meaning that $\OO$ contains a maximal order in a commutative subalgebra of $B$, and being basic is a local property \cite[Proposition 24.5.10]{Voight}: this fact was originally proved by Eichler~\cite[Satz 8]{Eic36} for quaternion algebras over $\QQ$, and generalized in~\cite{CSV}. This allows us to prove that an order is Bass by producing, for each prime $\ell$, an imaginary quadratic order $R$ in $\OO$ whose conductor is coprime to $\ell$.

\subsection{Computing in finite fields and quaternion algebras}

\subsubsection{Algebraic operations over \texorpdfstring{$\FF_{p^2}$}{Fp2}}\label{subsec:complexities}
We will state the complexity of our algorithms in terms of bit operations. Let $\llog x$ denote $\log\log x$.  Because supersingular elliptic curves and their isogenies may all be defined over $\FF_{p^2}$, we  record here the bit complexity of various algebraic operations over $\FF_{p^2}$. Let $\M(n)$ denote the bit-complexity of multiplying two $n$-bit integers. Then $\M(n)=O(n\log n)$~\cite{HarveyH21}. Let $a,b\in \FF_{p^2}$. We can compute the sum $a+b$, the product $ab$, and (when $a\not=0$) the inverse $a^{-1}$ in $O(\log p)$, $O(\M(\log p))=O(\log p(\llog p))$, and $O(\M(\log p)\llog p) = O(\log p(\llog p)^2)$  bit operations respectively, see~\cite[Corollary 9.9, Theorem 8.27, Corollary 11.11]{vzGG}. For a polynomial $f\in \FF_{p^2}[x]$ we can compute  the irreducible factors of $f$ in $\FF_{p^2}[x]$ and their multiplicities in expected $O(d\M(d)\log(pd)\M(\log p)) = O(d^2(\log d)(\log pd)(\log p)(\llog p))$ bit operations~\cite[Theorem 14.14]{vzGG}. 

\subsubsection{Computing in quaternion algebras}
We will often require algorithms to take an order in a quaternion algebra as an input, or provide one as an output. We represent a quaternion algebra $H(a,b)$ by the rational numbers $a,b$. The {\em size} of a rational number $m/n$ with $\gcd(m,n)=1$ is the number of bits required to specify the integers $m$ and $n$ and therefore $\size(m/n)=O(\max\{\log_2(m),\log_2(n)\})$.   We represent elements of $H(a,b)$ as $\QQ$-linear combinations of the symbols $1,i,j,ij$ and use the multiplication rules $i^2=a,j^2=b$, $ij=-ji$. The {\em size} of $H(a,b)$ is the number of bits required to represent the multiplication table for the basis $1,i,j,ij$, so $\size(H(a,b))=O(\max\{\size(a),\size(b)\})$. Given a vector $v\in \QQ^4$, define $\size(v)$ be to be the sum of the sizes of its coefficients. We represent an order in $\OO$ by  four vectors $v_1,v_2,v_3,v_4\in \QQ^4$ which are the coefficient vectors of a basis of $\OO$ in terms of the basis $1,i,j,ij$.
The {\em size} of a $\ZZ$-basis $\{v_1,v_2,v_3,v_4\}$ for $\OO$ is the size of $H(a,b)$  plus $\sum\size(v_i)$. We will often abuse notation and write $\OO$ as the input or output to an algorithm; by this we mean a basis of $\OO$ is the input or output. In this context we will also write $\size(\OO)$ for the size of the input or output basis. Various other integer quantities capturing the size of $\OO$, such as its (reduced) discriminant, have size polynomial in the size of a suitable basis of $\OO$.

\section{Inseparable endomorphisms}\label{sec:inseparables}

Let $E$ be a supersingular elliptic curve defined over $\mathbb F_{p^2}$ and let $\alpha\in \End(E)$. We say that $\alpha$ is \emph{inseparable} if $\alpha=\pi\circ \phi$, where $\phi\in \Hom(E,E^{(p)})$. The set of inseparable endomorphisms $P:=\pi\Hom(E,E^{(p)})$ is a 2-sided ideal  of $\End(E)$ and we refer to it as the \emph{ideal of inseparable endomorphisms of $E$}. 

In this section, we first study the arithmetic properties of $\mathbb Z+P\subseteq \End(E)$. Then in Subsection 3.2, we focus our attention on a particular kind of inseparable endomorphisms that we call \emph{inseparable reflections}.

\subsection{Properties of \texorpdfstring{$\mathbb Z+P$}{Z+P}}

    For completeness, we present the results of this subsection in the more general setting where  $B$ is a quaternion algebra over $\QQ$ ramified at a prime $p$.

\begin{prop}\label{prop:z+p}
Let $B$ be a quaternion algebra over $\QQ$ ramified at a prime $p$. Let $\OO$ be a maximal order in $B$ and let $P$ be the $2$-sided ideal in $\OO$ of reduced norm $p$. Then $\ZZ+P$ is a suborder of $\OO$ of index $p$, and $\OO$ is the unique maximal order of $B$ containing $\ZZ+P$.
\end{prop}

\begin{proof}
 We begin by showing that $\ZZ+P$ is an order. First, it is a lattice since it is finitely generated and $B=P\QQ \subseteq (\ZZ+P)\QQ$. 
 Second, since $P$ is an ideal, $\ZZ+P$ is closed under multiplication and contains $1\in B$ so $\ZZ+P$ is a subring of $B$. Therefore $\ZZ+P$ is a suborder of $\OO$. 
 
 We now calculate the index of $\ZZ+P$ in $\OO$.  Let $D= \disc(B)$. Since $P$ is invertible (as it is an integral ideal of a maximal order, see~\cite[Proposition 16.1.2]{Voight}), by~\cite[Proposition 16.7.7(iv)]{Voight}, we conclude $[\OO:P]=\Nrd(P)^2 = p^2$. 
Since $\ZZ\cap P\cong p\ZZ$ by~\cite[18.2.7(b)]{Voight}, as $\ZZ$-modules we have  $(\ZZ+P)/P\cong \ZZ/(\ZZ\cap P)\cong\ZZ/p\ZZ$. Therefore, $[\ZZ+P:P]=p$. By multiplicativity of the index, we have $[\OO:\ZZ+P]=p$,  so \cite[Lemma 15.2.15]{Voight} implies
\[
\disc(\ZZ+P)=[\OO:\ZZ+P]^2\disc(\OO)=p^2D^2=(pD)^2.
\]

Now we show that $\OO$ is the only maximal order containing $\ZZ+P$.  First, an order $\Lambda$ in $B$ is maximal at a prime $\ell\neq p$ if and only if $v_{\ell}(\discrd(\Lambda)) = v_{\ell}(D)$\cite[Lemma 15.5.3, Example 15.5.4]{Voight}.
Since the reduced discriminant of $\ZZ+P$ is $pD$,  
we have
$v_{\ell}(\discrd(\ZZ+P))=v_{\ell}(p)+v_{\ell}(D)=v_{\ell}(D)$, so the order $\ZZ+P$ is maximal at any prime $\ell\not=p$. This implies $\OO\otimes \ZZ_{\ell}=(\ZZ+P)\otimes \ZZ_{\ell}$ for any $\ell\not=p$. Moreover, since $B$ is ramified at $p$, by~\cite[Lemmas 10.4.3, 13.3.4]{Voight}, $\OO\otimes\ZZ_p$ is the unique maximal order of $B\otimes \QQ_p$ and contains $(\ZZ+P)\otimes \ZZ_p$. 
Therefore, for every prime $\ell$, $\OO\otimes \ZZ_{\ell}$ is the unique maximal $\ZZ_{\ell}$-order containing $(\ZZ+P)\otimes\ZZ_{\ell}$. By~\cite[Corollary 9.4.7, Theorem 9.4.9, Lemma 9.5.3]{Voight}, we conclude that $\OO$ is the unique maximal order containing $\ZZ+P$.
\end{proof}

\begin{rmk}
  The order $\ZZ+P$ is not hereditary \cite[Definition 21.4.1]{Voight}, since its reduced discriminant is divisible by $p^2$ and therefore is not square-free \cite[Lemma 23.3.18]{Voight}. It is not Eichler \cite[Definition 23.4.1]{Voight}, since it fails to be Eichler at $p$ (it is not maximal at $p$, and the only Eichler order in a local division quaternion algebra is the unique maximal order). The order $\ZZ+P$ is Bass, as its reduced discriminant is $pD$ and thus cubefree~\cite[Exercise 24.6.7(a)]{Voight}. However, the order $\ZZ+P$ is residually ramified at $p$ since $(\ZZ+P)/P\cong \ZZ/p\mathbb Z$ (see~\cite[24.3.2]{Voight} for a definition of {\em residually ramified}). Finally, the order $\ZZ+P$ is the {\em order of level $p^2$} in its  unique maximal overorder (see~\cite[Definition 3.5]{Piz80theta}).
\end{rmk}

\begin{rmk}Let $E/\FF_{p^2}$ be a supersingular elliptic curve. To compute a basis of $\End(E)$, one can first compute  a basis of $\ZZ+P\subseteq\End(E)$ and then use Algorithms 7.9 and 3.12 in \cite{V2013} to recover a basis of the unique maximal order $\mathcal O$ containing $\ZZ+P$. In fact Proposition~\ref{prop:z+p} implies $\OO=\End(E)$. We refer the reader to section~\ref{sec:voight-algorithm} of the Appendix for algorithmic aspects of recovering $\End(E)$ from $\ZZ+P$. 
\end{rmk}

\subsection{Inseparable reflections}
We now define, inside the ideal of inseparable endomorphisms of $E$, the \emph{inseparable reflections}. These are inseparable endomorphisms whose construction is based on a symmetry of the supersingular $\ell$-isogeny graph $G(p,\ell)$ given by the Galois involution (see Subsection \ref{subsec:insref} for a formal definition).

\subsubsection{The Galois involution of \texorpdfstring{$G(p,\ell)$}{Gpl}}\label{subsec:Galoisinvolution}

Let $\sigma_p\colon \FF_{p^2}\to \FF_{p^2}$ be the $p$-power Frobenius automorphism such that $\sigma_p(\alpha)=\alpha^p$, for $\alpha\in\FF_{p^2}$. The Galois group $\Gal(\FF_{p^2}/\FF_p)=\langle\sigma_p\rangle$ acts on the set of  elliptic curves defined over $\FF_{p^2}$ sending $E$ to $E^{(p)}$. Note that $(E^{(p)})^{(p)}=E$ and that $E^{(p)}=E$ if and only if $E$ is defined over $\FF_p$.

Similarly we can define an action of $\Gal(\FF_{p^2}/\FF_p)$ on separable isogenies defined over $\FF_{p^2}$. Given a rational function $f\in \FF_{p^2}(x,y)$, let $f^{(p)}$ denote the rational function obtained by raising the coefficients of $f$ to the $p$-th power. Given a separable isogeny $\phi\colon E_1\to E_2$ defined over $\FF_{p^2}$, let us choose representative 
coordinate functions $f,g\in \FF_{p^2}(E_1)$, defined on $E_1- \ker \phi$, so that $\phi(x,y) = (f(x,y), g(x,y))$. Therefore, $\sigma_p$ maps $\phi$ to the isogeny $\phi^{(p)}\colon E_1^{(p)}\to E_2^{(p)}$ such that
$\phi^{(p)}(x,y) = (f^{(p)}(x,y), g^{(p)}(x,y))$. 
It is easy to see that the kernel of $\phi^{(p)}$ is $\pi(\ker\phi)$. Moreover, we have $(\phi^{(p)})^{(p)}=\phi$.

\begin{lem}\label{lem:phip_props}
Let $E_1,E_2,$ and $E_3$ be  elliptic curves 
defined over $\FF_{p^2}$, and let $\phi_1: E_1\to E_2$ and $\phi_2: E_2\to E_3$ be separable isogenies defined over $\FF_{p^2}$. The following  hold.
\begin{enumerate}[(a)]
\item\label{lem:PhipComposition} $(\phi_2\circ\phi_1)^{(p)} = \phi_2^{(p)}\circ\phi_1^{(p)}$. 
\item\label{phip_cd} $\phi_1^{(p)}\circ \pi = \pi\circ \phi_1$.
\item\label{pdualp} $(\widehat{\phi_1^{(p)}})^{(p)}=\widehat{\phi_1}$. Equivalently, $\widehat{\phi_1}^{(p)}=\widehat{\phi_1^{(p)}}$. 
\end{enumerate}
\end{lem}
\begin{proof}
Part~(\ref{lem:PhipComposition}) follows from the calculation that for functions
$f,g,h\in \FF_{p^2}(x,y)$, we have 
\[
(f(g(x,y),h(x,y)))^{(p)} = f^{(p)}(g^{(p)}(x,y),h^{(p)}(x,y)).
\]
Next, we prove~(\ref{phip_cd}). Let us choose representative 
coordinate functions $f,g$ so that $\phi_1(x,y) = (f(x,y), g(x,y))$. Then, $\phi_1^{(p)}(x,y) = (f^{(p)}(x,y), g^{(p)}(x,y))$. This implies 
\begin{align*}
(\phi_1^{(p)}\circ \pi)(x,y) &= \phi_1^{(p)}(x^p,y^p)  \\
&=(f^{(p)}(x^p,y^p),g^{(p)}(x^p,y^p)) \\
&= ((f(x,y))^p, (g(x,y))^p) \\
&= (\pi\circ \phi_1)(x,y). 
\end{align*}

\noindent We now prove~(\ref{pdualp}). We compute 
\begin{align*}
(\widehat{\phi_1^{(p)}})^{(p)} \circ \phi_1 &= (\widehat{\phi_1^{(p)}})^{(p)} \circ (\phi_1^{(p)})^{(p)} =
((\widehat{\phi_1^{(p)}})\circ \phi_1^{(p)})^{(p)}\\ &= ([\deg\phi_1^{(p)}]_{E_1^{(p)}})^{(p)} = ([\deg\phi_1]_{E_1^{(p)}})^{(p)}\\ 
&=[\deg\phi_1]_{E_1},
\end{align*}
where the first equality follows since $\phi_1$ is defined over $\FF_{p^2}$, in the second equality we used part~(\ref{lem:PhipComposition}), and in the fourth one we used $\deg\phi_1 = \deg\phi_1^{(p)}$. The last equality follows from the fact that coordinate functions for the multiplication-by-$m$ map on a curve $E$ are determined by $\psi_{E,m}$, the $m$th division polynomial of $E$~\cite[Exercise 3.7]{AEC}, along with the observation that the recursive definition of $\psi_{E,m}$ implies $\psi^{(p)}_{E,m} = \psi_{E^{(p)},m}$. 
Therefore $\widehat{\phi_1} = (\widehat{\phi_1^{(p)}})^{(p)}.$ 
\end{proof}

Because every $\overline{\mathbb F_p}$-isomorphism class of supersingular elliptic curves contains a model defined over $\mathbb F_{p^2}$ such that all the isogenies are also defined over $\mathbb F_{p^2}$, the Frobenius automorphism $\sigma_p\in\Gal(\FF_{p^2}/\FF_p)$ induces an automorphism of order $2$, i.e. an involution, of $G(p,\ell)$.  In particular, for every $\ell$-isogeny $\phi:E_1\rightarrow E_2$ there is the $\ell$-isogeny $\phi^{(p)}:E_1^{(p)}\rightarrow E_2^{(p)}$. The fixed vertices of this automorphism  correspond to supersingular curves defined over $\FF_p$, and following the terminology of \cite{ACLLNSS}, this action can be visualized as a reflection of $G(p,\ell)$ over the {\em spine} consisting of curves defined over $\mathbb F_p$.
Going forward, in order to lighten the notation, we write $\psi\phi$, instead of $\psi\circ\phi$, for the composition of two (or more) isogenies.

\subsubsection{Arithmetic properties of inseparable reflections}\label{subsec:insref}
In order to define inseparable reflections we introduce the concept of $(d,\epsilon)$-structures, defined by Chenu and Smith in \cite{CS21} (see also the notion of {\em $d$-admissable} curves in~\cite{morain:hal-01320388}).
\begin{defn}\label{def:depsilon}
 Let $d$ be a positive integer coprime to $p$. A {\em $(d,\epsilon)$-structure} is a pair $(E,\psi)$ where $E$ is an elliptic curve defined over $\FF_{p^2}$ and $\psi\colon E\to E^{(p)}$ is a degree $d$-isogeny satisfying 
 $\psi^{(p)} = \epsilon \widehat{\psi}$ with $\epsilon \in \{\pm 1\}$.  We say that $(E,d)$ is {\em supersingular} if $E$ is supersingular. 

\end{defn}

A $(d,\epsilon)$-structure $(E,\psi)$ yields an endomorphism $\mu=\pi\psi$ of $E$, which Chenu and Smith call its {\em associated endomorphism}.
When $d$ is square-free, a supersingular $(d,\epsilon)$-structure $(E,\psi)$  yields an associated endomorphism $\mu=\pi\psi$ of $E$ such that 
$\ZZ[\mu]\cong \ZZ[\sqrt{-dp}]$, \cite[Proposition 2]{CS21}.  In fact, this holds for arbitrary $d$ coprime to $p$, assuming $p>3$.
\begin{prop}\label{prop:associated_trace}
    Let $d$ be an integer coprime to a prime $p>3$, and let $(E,\psi)$ be a 
    $(d,\epsilon)$-structure. If $\mu=\pi\psi$ is the associated endomorphism of $E$, then $\mu^2=[-dp]$ and $\pi_E=-\epsilon p$.
\end{prop}
\begin{proof}
    The argument is similar to those in Propositions 1 and 2 of~\cite{CS21}. First, since $(E,\psi)$ is a $(d,\epsilon)$-structure, we have $\psi^{(p)}=\epsilon\widehat{\psi}$. 
    Therefore, 
    \[
    \mu^2=\pi\psi\pi\psi = \pi\pi\psi^{(p)}\psi = \pi_E\epsilon \widehat{\psi}\psi = \epsilon d\pi_E.
    \]
    Let $x^2-ax+dp$ be the characteristic polynomial of $\mu$. We now show $a=0$. Suppose toward a contradiction that $a$ is nonzero. We have $a\mu = \mu^2+dp=\epsilon d\pi_E+dp$. 
    Taking traces, we have 
    \[
    a^2=\Trd(a\mu) = \Trd(\epsilon d\pi_E+dp) = \epsilon d\Trd\pi_E +2dp.
    \]
    We first observe that this implies $d|a^2$. Since $E$ is supersingular, we have $p|\Trd\pi_E$, so we conclude $p|a^2$, and since $p$ is prime, $p^2|a^2$ as well. Since $p$ and $d$ are coprime, $dp^2$ divides $a^2$. Since we assume $a$ is nonzero, we obtain $dp^2\leq a^2$. On the other hand, $\ZZ[\mu]$ must have non-positive discriminant, so $a^2-4dp<0$. Thus
    \[
    dp^2 \leq  a^2 \leq 4dp ,
    \]
    which implies $p<4$. This is our desired contradiction, so we conclude $a=0$ and $\mu^2=-dp$. Finally, we have $0=\epsilon d\Trd\pi_E+2dp$, which implies $\Trd\pi_E = -2\epsilon p$. This implies $\pi_E=-\epsilon p$. 
\end{proof}

We now discuss a construction of a $(d,\epsilon)$-structure for $d$ which is not necessarily square-free. 
\begin{prop}\label{prop:reflection}
    Let $E_1$ be a supersingular elliptic curve. If 
$\phi\colon E_1\to E_2$ is a  $d_1$-isogeny and  $(E_2,\psi)$ is a $(d,\epsilon)$-structure, then 
$(E_1,\widehat{\phi^{(p)}}\psi\phi)$ is a $(d_1^2d,\epsilon)$-structure.
\end{prop}

\begin{proof}
    We must show $(\widehat{\phi^{(p)}}\psi\phi)^{(p)} = \epsilon\widehat{\widehat{\phi^{(p)}}\psi\phi}$:
    \begin{align*}
        (\widehat{\phi^{(p)}}\psi\phi)^{(p)} &= 
        (\widehat{\phi^{(p)}})^{(p)}\psi^{(p)}\phi^{(p)} 
        & \text{by Lemma \ref{lem:phip_props}, part~(\ref{lem:PhipComposition}}) \\ 
        &= \widehat{\phi}\psi^{(p)} \phi^{(p)} & \text{by Lemma \ref{lem:phip_props} part (\ref{pdualp}})\\
        &= \widehat{\phi}\epsilon \widehat{\psi} \phi^{(p)} & (E,\psi)\text{ is a } (d,\epsilon)-\text{structure}\\
        &= \epsilon \widehat{\widehat{\phi^{(p)}}\psi \phi}. &
    \end{align*}
\end{proof} 
 Below, we define a special type of associated endomorphism to a $(d,\epsilon)$-structure. We call these endomorphisms {\em inseparable reflections} since they arise from paths in isogeny graphs whose image under the Galois involution is the same path, traversed in the opposite direction. 

\begin{defn}\label{def:insepref}
    Let $p$ be a prime, and let $d_1,d$ be coprime integers, with $d$ square-free, which are both coprime to $p$. An {\em inseparable reflection} of degree $d_1^2dp$ of a supersingular elliptic curve $E_1$ defined over $\FF_{p^2}$ is an endomorphism 
    \[
    \alpha = \pi\widehat{\phi^{(p)}}\psi\phi
    \]
    such that $\phi\colon E_1\to E_2$ is a cyclic $d_1$-isogeny,  $(E_2,\psi)$ is a $(d,\epsilon)$-structure, and $\phi$ does not factor nontrivially through an isogeny $\phi'\colon E_1\to E_2'$  such that $E_2'$ has a $(d,\epsilon)$-structure $(E_2',\psi')$. 
\end{defn}

We now study the arithmetic of orders generated by inseparable reflections. First, we determine the imaginary quadratic order generated by a single inseparable reflection, then we study orders generated by two or more inseparable reflections. In particular, we give sufficient conditions for when two inseparable reflections do not commute and hence generate a quaternionic suborder of $\End(E)$. The following proposition follows immediately from Propositions~\ref{prop:associated_trace} and~\ref{prop:reflection}.

\begin{prop}\label{prop:trace_zero_insep}
 Let $E$ be a supersingular elliptic curve defined over $\FF_{p^2}$, and let $\alpha=\pi\widehat{\phi^{(p)}}\psi\phi$ be an inseparable reflection of degree $d_1^2dp$.  Then $\alpha^2=[-d_1^2dp]$ and in particular $\alpha$ has trace zero. 
\end{prop}

We show in Lemma~\ref{lem:reflect cyclic}  that the kernel of an inseparable reflection is cyclic. For this, we need the following lemma. This will be needed in~\ref{sec:bass} when we study orders generated by two or more inseparable reflections.

\begin{lem}\label{lem:cyclic composition}
Let $E_1,E_2$, and $E_3$ be elliptic curves defined over $\FF_q$ and let $\phi_1\colon E_1\to E_2$ and $\phi_2\colon E_2\to E_3$ be separable, cyclic isogenies. Then $\ker(\phi_2\phi_1)$ is cyclic if and only if $\ker\widehat{\phi_1}\cap \ker\phi_2$ is trivial. 
\end{lem}

\begin{proof}
If $\ker\widehat{\phi_1}\cap\ker\phi_2 = G$ is nontrivial, let $\tau\colon E_2\to E'$ be a separable isogeny 
with kernel $G$, where $E'$ is an elliptic curve defined over $\FF_q$. Then, both $\widehat{\phi_1}$ and $\phi_2$ factor through $\tau$: there exist isogenies $\psi_1,\psi_2$ such that $\widehat{\phi_1} = \psi_1 \tau$ and $\phi_2 =\psi_2 \tau$. 
\begin{center}
 \begin{tikzcd}
E_1  \arrow[r, "\phi_1"] & E_2 \arrow[r,"\phi_2"] \arrow[d,"\tau"] & E_3  \\
& E'  \arrow[ur, "\psi_2" below] \arrow[ul, "\psi_1"]& \\
\end{tikzcd}
\end{center}

Then, 
\[
\phi_2 \phi_1 = \psi_2  \tau  \widehat{\tau}  \widehat{\psi_1} = \psi_2 \widehat{\psi_1} [\# G]
\]
does not have cyclic kernel. 

Now assume that $\ker(\phi_2 \phi_1)$ is not cyclic. Let $S\in E_2(\overline{\FF_q})$ such that $\ker \phi_2=\langle S\rangle$, the cyclic group generated by $S$, and let $Q\in E_1(\overline{\FF_q})$ such that $\phi_1(Q) = S$. 
Also let $P\in E_1(\overline{\FF_q})$ such that  $\langle P\rangle=\ker \phi_1$. 

First, we claim that $\ker(\phi_2\phi_1) = \langle P\rangle+\langle Q\rangle$. 
Let $P'\in \ker(\phi_2\phi_1)$. Then, $\phi_1(P')=[a]S$ for some $a$. Therefore, $P'-[a]Q\in \ker \phi_1$. Thus, 
\[
P' = (P'-[a]Q) + [a]Q \in \ker\phi_1 + \langle Q \rangle=\langle P \rangle+\langle Q \rangle,
\]
i.e. $\ker(\phi_2\phi_1) \subseteq \langle P\rangle + \langle Q \rangle$.
Since $\phi_1\left(\langle P\rangle + \langle Q\rangle\right)\subseteq \ker\phi_2$, we also have that 
$\ker(\phi_2\phi_1) \supseteq \langle P\rangle + \langle Q \rangle$. Thus, $\ker(\phi_2\phi_1) = \langle P\rangle+\langle Q\rangle$. 

Since we assume that $\ker(\phi_2\phi_1)$ is not cyclic, $\langle P \rangle + \langle Q \rangle$ contains $E_1[d]$ for some $d>1$. Note that $d$ and $\deg\phi_1$ are not coprime, since otherwise 
$\phi_1(E_1[d])=E_2[d]$ and thus $E_2[d]\subseteq \ker\phi_2$, contradicting the assumption that $\ker\phi_2$ is cyclic. Let $g=\gcd(d,\deg\phi_1)$. Then, $E_1[g]\subseteq E_1[d]$ and $E_1[g]\subseteq E_1[\deg\phi_1]$. Now we have that $\phi_1(E_1[g])\subseteq \ker \phi_2$ and also 
$\phi_1(E_1[g])\subseteq \ker\widehat{\phi_1} = \phi_1(E_1[\deg\phi_1])$, therefore $\phi_1(E_1[g])\subseteq  \ker\widehat{\phi_1} \cap \ker \phi_2$. Since $\phi_1$ is cyclic and $g>1$, 
$\phi_1(E_1[g])\not=0$, so $\ker\widehat{\phi_1}\cap \ker\phi_2\not=0$. 
\end{proof}

\begin{lem}\label{lem:reflect cyclic}
Let $E_1$ be a supersingular elliptic curve defined over $\FF_{p^2}$ and let 
$\alpha=\pi\widehat{\phi^{(p)}}\psi\phi$ be an inseparable reflection of degree $d_1^2dp$. Then the kernel of $\alpha$ is cyclic. 

\end{lem}

\begin{proof}
It suffices to show that $\widehat{\phi^{(p)}}\psi\phi\colon E_1\to E_1^{(p)}$ has cyclic kernel. 
Assume that $\ker\widehat{(\phi^{(p)}}\psi \phi)$ is not cyclic. Let $E_2$ be the codomain of $\phi$. We show that there is an isogeny $\tau\colon E_2\to E_3$ such that $\ker \tau\subseteq \ker\widehat{\phi}$ and $E_3$ has a $(d,\epsilon)$-structure. By Lemma~\ref{lem:cyclic composition}, we have that $G=\ker\widehat{\phi}\cap \ker\widehat{\phi^{(p)}}\psi\not=0$. Note that $G$ is defined over $\FF_{p^2}$, since it is contained in $\ker\widehat{\phi}$ which is defined over $\FF_{p^2}$. Let $\tau\colon E_2\to E_3$ be an isogeny defined over $\FF_{p^2}$ with kernel $G$.

\begin{center}
 \begin{tikzcd}
E_1  \arrow[r, "\phi"] \arrow[d, "\pi" left] & E_2 \arrow[r,"\tau"] \arrow[d,"\psi"] & E_3  \\
 E_1^{(p)}  \arrow[r, "\phi^{(p)}" ] & E_2^{(p)} &  \\
\end{tikzcd}
\end{center}

We show that $E_3$ has an $(d,\epsilon)$-structure. By \cite[Lemma~1]{CS21}, it suffices to show that $\End(E_3)$ contains a quadratic order isomorphic to $\ZZ[\sqrt{-dp}]$.

First, we claim that $\pi\psi(G) = G$. Since $G\subseteq \ker(\widehat{\phi^{(p)}}\psi)$, we have that 
\[
\psi(G) \subseteq \ker\widehat{\phi^{(p)}} = \pi(\ker\widehat{\phi}).
\] 
Since $\gcd(d_1,d)=1$, we see that $\psi$ 
induces an isomorphism $E_2[d_1]\to E_2^{(p)}[d_1]$. Thus, since $G\subseteq E_2[d_1]$, we have $\#\psi(G) = \#G$. Moreover, $\ker\widehat{\phi}$ is cyclic, so $\pi(\ker\widehat{\phi})$ is also cyclic. Therefore, $\psi(G)$ is the 
unique subgroup of $\pi(\ker\widehat{\phi})$ of order $\#G$. Since the unique 
subgroup of $\ker\widehat{\phi}$ of order $\#G$ is also $G$, we have
\[
\psi(G) = \pi(G).
\]
From this we conclude that 
\[
\pi\psi(G) = \pi(\pi(G)) =G,
\]
where the last equality holds since $\tau$ is defined over $\FF_{p^2}$. Therefore the proof of the claim is complete. 

Now consider the endomorphism 
\[
\rho = \tau\pi\psi\widehat{\tau}\in \End(E_3).
\]
We claim that $\rho(E_3[\deg\tau])=0$. Indeed, 
\[
\rho(E_3[\deg\tau]) = \tau\pi\psi\widehat{\tau}(E_3[\deg\tau]) = \tau\pi\psi(\ker \tau) = \tau\pi\psi(G) = \tau(G) = 0.
\]
Thus, $\mu=\frac{1}{\deg\tau}\rho$ is an endomorphism of $E_3$. Observe that 
\[
\mu^2 = \frac{1}{(\deg\tau)^2} \tau\pi\psi\widehat{\tau}\tau\pi\psi\widehat{\tau} = \frac{1}{\deg\tau} \tau\pi\psi\pi\psi\widehat{\tau} = \frac{-dp}{\deg\tau}\tau\widehat{\tau} = -dp,
\]
so $\ZZ[\mu]\cong \ZZ[\sqrt{-dp}]$. As mentioned above, by Lemma 1 of~\cite{CS21}, it follows that $E_3$ has a $(d,\epsilon)$-structure (indeed, $\mu=\pi\psi'$ for an isogeny $\psi'\colon E_3\to E_3^{(p)}$, and $(E_3,\psi')$ is the desired $(d,\epsilon)$-structure). 

\end{proof}

\subsection{Quaternionic suborders of \texorpdfstring{$\End(E)$}{EndE} generated by inseparable reflections}\label{sec:bass}

In this section we study orders generated in $\End(E)$ by two inseparable reflections. The main result in this section, Theorem~\ref{thm:bass_from_inseps}, shows that assuming some mild restrictions on their degrees, two inseparable reflections generate a Bass suborder of $\End(E)$. First, we use Lemma \ref{prop:trace_zero_insep}, Lemma \ref{lem:cyclic composition} and Lemma \ref{lem:reflect cyclic} to give sufficient conditions for two inseparable endormophisms to not commute and therefore to generate a quaternionic suborder of $\End(E)$.
\begin{thm}\label{thm:genorder}
Let $E$ be a supersingular elliptic curve defined over $\FF_{p^2}$, and let $\alpha_1=\pi \widehat{\phi_1^{(p)}}\psi_1 \phi_1$ and $\alpha_2=\pi\widehat{\phi_2^{(p)}}\psi_2 \phi_2$ be inseparable reflections of degree $d_1^2dp$ and $d_2^2dp$.  If $\ker\phi_1\not=\ker\phi_2$, then  $\alpha_1$ and $\alpha_2$  do not commute. 
\end{thm}
\begin{proof}
 Assume that $\alpha_1$ and $\alpha_2$ commute. Then,
$\QQ(\alpha_1) = \QQ(\alpha_2)$, so there exist integers 
$k,m,n$ such that $[k]\alpha_1=[m]+[n] \alpha_2$. By Lemma~\ref{prop:trace_zero_insep}, we have 
$\Trd(\alpha_1)=\Trd(\alpha_2)=0$, so $m=0$ and $[k]\alpha_1=[n] \alpha_2$. 
We claim that $k|n$. Write $n=kq+r$ with $0\leq r<k$. Note that 
since $[n](\alpha_2(E[k]))=0$, we also have $[r](\alpha_2(E[k]))=0$. This implies 
$\alpha_2\left(E\left[\frac{k}{\gcd(k,r)}\right]\right) = 0$. The kernel of $\alpha_2$ is cyclic by 
Lemma~\ref{lem:reflect cyclic}, so we must have that $k/\gcd(k,r)=1$ and 
hence $\gcd(k,r)=k$ implying $r=0$. Thus $k|n$. Therefore 
$\alpha_1=[n/k]\alpha_2$. Now, since $\alpha_1$ has cyclic kernel by 
Lemma~\ref{lem:reflect cyclic}, we conclude $n/k=\pm 1$. Thus $\alpha_1=\pm 
\alpha_2$ so $\ker\alpha_1 = \ker\alpha_2$ and $\deg\phi_1 = \deg\phi_2$. Therefore, using the property that $\ker\alpha_i$ is cyclic for $i=1,2$, we obtain  $\ker\phi_1$ = $\ker\phi_2$.
 \end{proof}

 We now show that we can control arithmetic properties of an order generated by two inseparable reflections with mild assumptions on their degrees. 

\begin{prop}\label{prop:gorenstein}
Let $E$ be a supersingular elliptic curve defined over $\FF_{p^2}$, and let $d_1,d_2,d$ be three pairwise coprime integers, with $d$ square-free. For $i=1,2$, let $\alpha_i=\pi \widehat{\phi_i^{(p)}}\psi_i \phi_i\in\End(E)$  be an inseparable reflection of degree $d_i^2dp$.

\begin{enumerate}[(i)]
\item\label{prop:gor:order} The endomorphisms $1,\alpha_1,\alpha_2,\alpha_1\alpha_2$ generate an 
order 
\[\Lambda_{\alpha_1\alpha_2}:=\mathbb Z+\mathbb Z\alpha_1+\mathbb Z\alpha_2+\mathbb Z\alpha_1\alpha_2\subseteq \End(E).
\]
\item\label{prop:gor:disc} The endomorphism $\alpha_1\alpha_2$ factors through the multiplication-by-$p$ map, so 
\[
\rho\coloneqq\frac{-\alpha_1\alpha_2}{p} 
\]
is an endomorphism of $E$. The discriminant of $\Lambda_{\alpha_1\alpha_2}$ is 
\[
\disc(\Lambda_{\alpha_1\alpha_2}) = p^4\cdot((\Trd\rho)^2 - 4\deg\rho)^2 = p^4\cdot (\disc\rho)^2.
\]
\item\label{prop:gor:gor} The order  $\Lambda_{\alpha_1\alpha_2}$ is Gorenstein. 
\end{enumerate}
 
\end{prop}

\begin{proof}
 Lemma \ref{prop:trace_zero_insep} implies that $\alpha_1$ and $\alpha_2$ are non-scalar endomorphisms.  Since $d_1\not=d_2$, we have that $\ker\phi_1\not=\ker\phi_2$ so $\alpha_1\alpha_2\not=\alpha_2\alpha_1$ by Theorem~\ref{thm:genorder}. The endomorphisms $\alpha_1$ and $\alpha_2$ are noncommuting, nonscalar elements of $\End^0(E)$, so $\Lambda_{\alpha_1\alpha_2}$ is a lattice in $\End^0(E)$. Since $\alpha_1,\alpha_2,\alpha_1\alpha_2$ are integral, the lattice $\Lambda_{\alpha_1\alpha_2}$ is a ring containing $1$, so it is an order, completing the proof of part~(\ref{prop:gor:order}).
 
To prove part~(\ref{prop:gor:disc}) we compute $\discrd(\Lambda_{\alpha_1\alpha_2})$. Since $\Trd \alpha_i=0$, we have $\widehat{\alpha_i}=-\alpha_i$, so 
\[
\rho=\frac{-1}{p}\alpha_1\alpha_2 = \frac{1}{p}\widehat{\phi_1}\widehat{\psi_1}\widehat{\pi}\pi \phi_1^{(p)}\widehat{\phi_2^{(p)}}\psi_2 \phi_2 = \widehat{\phi_1}\widehat{\psi_1} \phi_1^{(p)}\widehat{\phi_2^{(p)}}\psi_2 \phi_2.
\]
The Gram matrix of the basis  $1,\alpha_1,\alpha_2,\alpha_1\alpha_2$  is 
\[
G\coloneqq\begin{pmatrix}
2 & 0 & 0 & -p\Trd\rho \\ 
0 & 2pdd_1^2 & p\Trd\rho & 0 \\
0 & p\Trd\rho & 2pdd_2^2 & 0 \\
-p\Trd\rho & 0 & 0 & 2(pd_1d_2d)^2\\
\end{pmatrix}.
\]
A calculation shows that its determinant, and therefore the discriminant of $\Lambda{\alpha_1\alpha_2}$, is 
\[ 
\det(G)= p^4\cdot((\Trd\rho)^2 - 4\deg\rho)^2 = p^4\cdot (\disc\rho)^2.
\]

Finally we prove part~(\ref{prop:gor:gor}). 
We claim that the ternary quadratic form attached 
to $\Lambda_{\alpha_1\alpha_2}$ is 
\[
Q(x,y,z)=pdd_2^{2}x^2+pdd_1^{2}y^2+z^2-tpxy,
\]
where $t=\Trd\rho$.  
 A calculation shows the basis $1,i=\alpha_1,j=\alpha_2,k=\alpha_2\alpha_1$ of $\Lambda_{\alpha_1\alpha_2}$ is a {\em good basis} in the sense of~\cite[22.4.7]{Voight}, i.e., there exist integers $a,b,c,u,v,w$ satisfying $i^2=ui-bc$, $j^2=vj-ac$, $k^2=wk-ab$ and $jk=a\widehat{i}$, $ki=b\widehat{j}$, and $ij=c\widehat{k}$. Given a good basis, the corresponding ternary quadratic form is $ax^2+by^2+cz^2+uyz+vxz+wxy$ (see the proof of~\cite[Proposition 22.4.12]{Voight}).  
In the case of the basis $1,\alpha_1,\alpha_2,\alpha_2\alpha_1$ for $\Lambda$, we have $a=pdd_2^2$, $b=pd_1^2$, $c=1$, $u=v=0$, and $w=-tp$. The quadratic form $Q$ is primitive since its  coefficients are coprime, so $\Lambda$ is Gorenstein by~\cite[Theorem 24.2.10]{Voight}. 
\end{proof}
\begin{rmk}\label{rmk:LambdaRho}
     The lattice $\Lambda_{\rho}:=\mathbb Z+\mathbb Z\alpha_1+\mathbb Z\alpha_2+\mathbb Z\rho$ in $\End(E)$ is also a suborder of $\End(E)$ and clearly $\Lambda_{\alpha_1\alpha_2}\subsetneq \Lambda_{\rho}$. 
The order $\Lambda_{\alpha_1\alpha_2}$ is non-maximal precisely at $p$ and the primes dividing the discriminant of $\rho$. Assuming that $p\nmid \disc(\rho)$, the order $\Lambda_{\rho}$ is the  unique $p$-maximal order containing $\Lambda_{\alpha_1\alpha_2}$ whose localizations at all $\ell\not=p$  agree with those of $\Lambda_{\alpha_1\alpha_2}$.
\end{rmk}
\begin{thm}\label{thm:bass_from_inseps}
Let $E$ be a supersingular elliptic curve defined over $\FF_{p^2}$, and let $d_1,d_2,d$ be three pairwise coprime integers, with $d$ square-free. For $i=1,2$, let \[
\alpha_i=\pi_p \widehat{\phi_i^{(p)}}\psi_i \phi_i\in\End(E)
\]
be an inseparable reflection of degree $d_i^2dp$. Finally, assume that $-dp\not\equiv 1\pmod{4}$. Then the order $\Lambda_{\alpha_1\alpha_2}$ is Bass. 
\end{thm}

\begin{proof}
Proposition~\ref{prop:gorenstein} implies that $\Lambda_{\alpha_1\alpha_2}$ is an order. 
We show that $\Lambda_{\alpha_1\alpha_2}$ is locally basic, and hence locally Bass by~\cite[Proposition 1.11]{Brz90} at every prime $\ell$. This suffices, since being Bass is a local property~\cite[Proposition 24.5.10]{Voight}. 

 Consider the quadratic order $R_i=\ZZ[\alpha_i]\cong \ZZ[d_i\sqrt{-dp}]$ in $\Lambda_{\alpha_1\alpha_2}\subseteq\End(E)$. Since $-dp$ is square-free and not congruent to $1$ modulo $4$, the maximal order in the fraction field of $R_i$ is isomorphic to $\ZZ[\sqrt{-dp}]$, so the conductor of $R_i$ is $d_i$. Then, for any prime $\ell$, at least one of $R_1$ or $R_2$ is maximal at $\ell$ since $R_1$ is maximal at every prime $\ell$ which does not divide $d_1$, and $R_2$ is maximal at every prime $\ell$ which does not divide $d_2$. This shows $\Lambda_{\alpha_1\alpha_2}$ is locally basic at each prime $\ell$.
\end{proof}

\section{Computing an order in \texorpdfstring{$\End(E)$}{EndE} with inseparable endomorphisms}\label{sec:algorithms}

By Proposition~\ref{prop:trace_zero_insep},  one inseparable reflection $\alpha_1$ of degree $d_1^2dp$ of a  supersingular elliptic curve $E$ defined over $\mathbb F_{p^2}$ generates an imaginary quadratic order of discriminant $-4d_1^2dp$. If $d_2$ is another integer coprime to $d_1$, if we assume $d$ is coprime to both $d_1$ and $d_2$, and let $\alpha_2$ be a $d_2^2dp$-inseparable reflection, then $\alpha_1$ and $\alpha_2$ generate an order $\Lambda_{\alpha_1\alpha_2} \coloneqq \langle \alpha_1,\alpha_2\rangle$ in $\End(E)$ which is Gorenstein by Proposition~\ref{prop:gorenstein}. If we assume $-dp\not\equiv 1\pmod{4}$ then $\Lambda_{\alpha_1\alpha_2}$ is Bass by Theorem~\ref{thm:bass_from_inseps}. Therefore, if we can compute inseparable reflections of $E$ for certain values of $d_1,d_2,$ and $d$, then we can compute orders with certain desirable arithmetic properties in $\End(E)$. We will make the results in Section~\ref{sec:inseparables} effective by giving an algorithm for computing one inseparable reflection and then another algorithm which uses inseparable reflections for computing a Bass order $\Lambda$ in $\End(E)$. 

Computing a $d_i^2dp$-inseparable reflection requires a  $d_i$-isogeny $\phi_i\colon E\to E_i$ where $E_i$ has a $d$-isogeny $\psi_i\colon E_i\to E_i^{(p)}$. Computing such a $d_i$-isogeny will be easiest when $d_i$ is smooth: if $d_i=\ell_i^{t_i}$ where $\ell_i$ is a small prime, then we can take random walks of length $t_i$ in the $\ell_i$-isogeny graph until finding a supersingular curve $E_i$ which is $d$-isogenous to $E_i^{(p)}$. The simplest choice for $d$ is $d=2$: in this case we have that $-dp\not\equiv 1\pmod{4}$ and it is easy to check whether $E_i$ is $2$-isogenous to $E_i^{(p)}$. We follow this strategy in Algorithm~\ref{alg:insependo}. We show that it correctly computes an inseparable reflection and analyze its complexity in Proposition~\ref{prop:specialendoruntime}. We conclude with Algorithm~\ref{alg:bass} which computes a Bass order in $\End(E)$ according to Theorem~\ref{thm:makebass}.

\subsection{Computing inseparable reflections}
We compute an inseparable reflection of degree $\ell^{2t}dp$ by taking random non-backtracking walks beginning at $E$ of length $t$, which correspond to cyclic $\ell^{t}$ isogenies $\phi\colon E\to E'$, until finding a $(d,\epsilon)$-structure $(E',\psi)$. The resulting inseparable reflection of $E$ is $\pi\widehat{\phi^{(p)}}\psi\phi$. In order to bound the expected runtime of this approach, we must consider the probability that a random non-backtracking walk of length $t$ terminates at a supersingular curve $E'$ with a $(d,\epsilon)$-structure. 

Let $p>3$ be a prime. We recall some notation for the supersingular $\ell$-isogeny graph $G(p,\ell)$ from Section~\ref{sec:background} and~\cite[Section 3]{secuer}. Let $V$ denote a complete set of 
representatives of isomorphism classes of supersingular elliptic curves over $\FF_{p^2}$.  Let $H$ be the $\CC$-vector space $H$ with basis $V$, let $\langle\cdot,\cdot\rangle$ be the inner product such that for $E,E'\in V$, we have $\langle E,E'\rangle=w_E$ if $E=E'$ and $0$ otherwise. Let $A$ denote the adjacency operator for $G(p,\ell)$.   A {\em walk} in the graph $G(p,\ell)$ is defined to be a sequence of edges $\phi_1,\phi_2,\ldots,\phi_k$ such that the codomain of $\phi_i$ is isomorphic to the domain of $\phi_{i+1}$. A walk {\em has no backtracking}, or is non-backtracking, if $\phi_{i+1}\not=u\widehat{\phi_i}$ for any automorphism $u$. Thus, non-backtracking walks in $G(p,\ell)$ beginning at $E\in V$ are in bijection with cyclic subgroups of $E[\ell^{\infty}]$.   Let $\mathcal{E} = \sum_{E\in V} w_E^{-1}E$, so $s = \frac{1}{\langle \mathcal{E},\mathcal{E}\rangle} \mathcal{E}$ is the stationary distribution for the random walk in $G(p,\ell)$. For a distribution $v=\sum_{E\in V} v_E E$ and a subset $X\subseteq V$, let $v(X)\coloneqq \sum_{E\in X}v_E$ denote the probability that a vertex sampled according to $v$ is an element of $X$. For two distributions $v,v'$ on $G(p,\ell)$, let 
\[
d_{TV}(v,v')= \sup_{X\subseteq V}|v(X)-v'(X)| = \frac{1}{2}|v-v'|_{1}
\]
denote the total variation distance between $v$ and $v'$. Let $P^{(t)}$ be the transition matrix for the non-backtracking random walk in $G(p,\ell)$ of length $t$; we remark that $P^{(t)}$ and $P^t$ are not the same matrix for $t>1$. Holding $\ell$ constant, the following proposition states that a non-backtracking walk of length $O(\log p)$ will land in a set $X\subseteq V$ with probability proportional to $\#X/\#V$.

\begin{prop}\label{prop:nbtmixing}
Let $E_0\in V$ and let $X\subseteq V$ be nonempty. If 
\[
t/2-\log_{\ell}\left(t+\frac{\ell-1}{\ell+1}\right) \geq \log_{\ell}\left(\frac{(p-1)^{3/2}}{24\sum_{E\in X}w_E^{-1}}\right),
\]
then a non-backtracking random walk of length $t$ beginning at $E_0$ lands in $X$ with 
probability at least 
\[
\frac{6}{p-1}\sum_{E\in X}w_E^{-1}.
\]
\end{prop}
\begin{proof}
  Let $v^{(t)}=P^{(t)}E_0$ be the probability distribution on $V$ resulting from a random non-backtracking walk of length $t$ beginning at $E_0$. 
    Then, by Theorem 11 of~\cite{secuer}, 
    \[
    |v^{(t)}(X)-s(X)| \leq d_{TV}(v^{(t)},s)\leq
    \frac{(p-1)^{1/2}}{4}\cdot \left(t+\frac{\ell-1}{\ell+1}\right)\cdot\ell^{-t/2}. 
    \]
    We have
    \[
    s(X) = \frac{12}{p-1}\sum_{E\in X} w_E^{-1}. 
    \]
    We see that if 
    \[
    t/2-\log_{\ell}\left(t+\frac{\ell-1}{\ell+1}\right) \geq \log_{\ell}\left(\frac{(p-1)^{3/2}}{24\sum_{E\in X}w_E^{-1}}\right),
    \]
    then 
    \[
    \frac{(p-1)^{1/2}}{4}\cdot \left(t+\frac{\ell-1}{\ell+1}\right)\cdot\ell^{-t/2} \leq \frac{6}{p-1}\sum_{E\in X}w_E^{-1},
    \]
    so 
    \[
    v^{(t)}(X) \geq \frac{6}{p-1}\sum_{E\in X}w_E^{-1},
    \]
    as desired. 
\end{proof}

We now bound the expected number of random walks beginning at $E$ which we take before finding a $(d,\epsilon)$-structure. Let $\llog x$ denote $\log\log x$.

\begin{prop}[GRH]\label{prop:problowerbound}
Assume GRH. Let $p>3$ be a prime, and let $E_0$ be a supersingular elliptic curve defined over $\FF_{p^2}$ such that $\#E_0(\FF_{p^2})=(p+\epsilon)^2$ for $\epsilon=\pm 1$. Let $\ell\not=p$ be a prime, and let $d=1$ or $2$. Let $X$ be a complete set of representatives for the collection of isomorphism classes of supersingular elliptic curves with a $(d,\epsilon)$-structure. If 
\[
t/2-\log_{\ell}\left(t+\frac{\ell-1}{\ell+1}\right)\geq \log_{\ell}\left(
\frac{(p-1)^{3/2}}{8}\right), 
\]
then a non-backtracking walk in $G(p,\ell)$ beginning at $E_0$  of length $t$  lands in $X$ with probability at least $\Omega\left(\frac{1}{\sqrt{p}\llog p}\right)$.  
\end{prop}
\begin{proof}
The set $X$ is nonempty because~\cite[Corollary 1]{CS21} implies that the number of isomorphism classes of $(d,\epsilon)$-structures is at least the class number of $\ZZ[\sqrt{-dp}]$. 
Since $X$ is nonempty, we may apply the trivial lower bound $\#X\geq 1$ to get 
 \[
 \sum_{E\in X}\frac{2}{\#\Aut(E)}\geq 1/3.
 \]
For $t$ satisfying the hypothesis in the proposition, we  conclude 
 \begin{align*}
t/2-\log_{\ell}\left(t+\frac{\ell-1}{\ell+1}\right)&\geq \log_{\ell}\left(
\frac{(p-1)^{3/2}}{8}\right)  \\ 
&\geq \log_{\ell}\left(
\frac{(p-1)^{3/2}}{24\sum_{E\in X}\frac{2}{\#\Aut(E)}}\right).
\end{align*}
By Proposition~\ref{prop:nbtmixing}, a non-backtracking walk beginning at $E_0$ lands in $X$ with probability at least
\[
\frac{6}{p-1}\cdot \sum_{E\in X} \frac{2}{\#\Aut(E)}\geq \frac{6}{p-1}(\#X-7/6).
\]
 
 Let $K=\QQ(\sqrt{-pd})$. By~\cite[Corollary 1]{CS21}, there are at least $h_{K}$ many  $(d,\epsilon)$-structures, up to $\mathbb{F}_{p^2}$-isomorphism. Since any given $E$ has at most $d+1$ $d$-isogenies, and since the number of distinct $\FF_{p^2}$-isomorphism classes of curves with the same $j$-invariant is at most $6$,  we have that 
\[
\#X \geq \frac{1}{6(d+1)}h_{K}. 
\]
Assuming the Generalized Riemann Hypothesis, 
\[
h_{K} = \Omega(\sqrt{pd}/\llog(pd))=\Omega(\sqrt{p}/\llog(p))
\]
by \cite[Theorem 1]{Lit28}. We conclude that a non-backtracking walk of length $t$ lands in $X$ with probability $\Omega\left((\sqrt{p}\llog(p))^{-1}\right)$. 
\end{proof}

\begin{rmk}
    We required a lower bound on the class group of an imaginary quadratic order in two places in the proof of Proposition~\ref{prop:problowerbound}: once to determine the length $t$ of a walk to guarantee good mixing, and once to extract a lower bound on the probability that a random walk ends in a given set. Since our later algorithms require an {\em effective} upper bound on $t$, we need effective lower bounds on the class group for the first part of the argument. One could use non-trivial effective lower bounds, but these would only yield sub-logarithmic improvements to the size of $t$. For simplicity we just use $1$. In the second part of the argument, we do not need an effective lower bound, since the constant is hidden in the big-$\Omega$. 
\end{rmk}

Next, we show that if $d<p/4$, a curve $E$ has a $(d,\epsilon)$-structure if and only if $E$ is $d$-isogenous to $E^{(p)}$. This holds  if and only if 
$\Phi_d(j(E),j(E)^p)=0$, giving us an efficient method for testing whether 
$E$ has a $(d,\epsilon)$-structure. The following lemma is an adaptation of~\cite[Lemma 6]{CGL2009}, and we include a proof for convenience.
\begin{lem}\label{lem:check_deps}
Let $E$ be a supersingular elliptic curve over
$\mathbb{F}_{p^2}$, and let $1<d<p/4$ be square-free and coprime to $p$. Then $E$ has a $(d,\epsilon)$-structure $(E,\psi)$ for $\epsilon\in \{\pm1\}$
if and only if $E$ is $d$-isogenous to $E^{(p)}$.
\end{lem}

\begin{proof}
    If $(E,\psi)$ is a $(d,\epsilon)$-structure, then $\psi\colon E\to E^{(p)}$ is a $d$-isogeny. Assume now that $E$ is $d$-isogenous to $E^{(p)}$, and let $\psi\colon E\to E^{(p)}$ be a $d$-isogeny. 
    We show that the characteristic polynomial of  $\mu=\pi\psi$ is $x^2+dp$: if this holds, then 
    we have an embedding $\ZZ[\sqrt{-dp}] \hookrightarrow \End(E)$ defined by sending $\sqrt{-dp}$ to $\mu=\pi\psi$, and \cite[Lemma 1]{CS21} then implies that $(E,\psi)$ is a $(d,\epsilon)$-structure. Since the degree of $\mu$ is 
    $(\deg\pi)(\deg\psi)=pd$, we only need to show that $\Trd\mu$ is zero. The ring $\ZZ[\mu]$ must be an imaginary quadratic order since 
    $\deg\mu$ is not a square, and $p$ does split in $\ZZ[\mu]$ since $E$ is supersingular. Thus the characteristic polynomial \[x^2-(\Trd\mu )x+pd\equiv x(x-\Trd\mu)\pmod{p}\]
    of $\mu$ cannot have distinct roots modulo $p$, so we must have $\Trd\mu\equiv 0 \pmod{p}$. Since the discriminant of $\mu$ is negative and using our assumption that $d<p/4$, we have  
    \[
    |\Trd\mu|<2\sqrt{pd}<p.
    \]
Thus $\Trd\mu=0$. 
\end{proof}

We now introduce our algorithm for computing inseparable reflections, Algorithm~\ref{alg:insependo}. 
\begin{algorithm}\label{alg:insependo}
\KwIn{A supersingular elliptic curve $E/\FF_{p^2}$ with $p>3$, a prime $\ell$, and an integer $d$, with $d<p/4$ square-free and coprime to $\ell$.}
\KwOut{An inseparable reflection $\alpha=\pi\widehat{\phi^{(p)}}\psi\phi\in \End(E)$ where $\phi\colon E\to E'$ is an $\ell^k$-isogeny (represented by a sequence of $\ell$-isogenies)  and $\psi\colon E'\to E'^{(p)}$ is a $d$-isogeny  such that $(E',\psi)$ is a $(d,\epsilon)$-structure.
}
\caption{Compute an inseparable reflection}
Compute the least integer $t$ such that $t/2-\log_{\ell}\left(t+\frac{\ell-1}{\ell+1}\right)\geq \log_{\ell}\left(
\frac{(p-1)^{3/2}}{8}
\right)$\label{step:walklength}\;
\Repeat{$E_t$ is $d$-isogenous to $E_t^{(p)}$}{
Compute a random, non-backtracking walk $W=\{\phi_1\colon E\to E_1,\ldots,\phi_t\colon E_{t-1}\to E_t\}$ in $G(p,\ell)$ of length $t$\label{step:nbtwalk}\;
}
Let $k=\min_{1\leq i\leq t}\{i:E_i \text{ is }d\text{-isogenous to } E_i^{(p)}\}$\;
Compute a $(d,\epsilon)$-structure $(E_k,\psi)$\;
\Return{\{$\phi_1,\ldots,\phi_k,\psi,\widehat{\phi_k}^{(p)},\ldots,\widehat{\phi_1}^{(p)},\pi\}$}
\end{algorithm}
We only need to run Algorithm \ref{alg:insependo} on inputs of the form $(E,\ell,2)$ (to compute Bass orders) and inputs $(E,\ell,1)$ (for our  heuristic algorithm described in Section \ref{sec:heuristic}), where $\ell$ is a fixed small prime, such as $2,3,$ or $5$. Thus in our complexity analysis below, we are treating $\ell$ and $d$ as constants. Similar results hold for square-free $d= O(\log p)$  and prime $\ell= O(\log p)$. 
Let $\M(n)$ denote the cost of multiplying two $n$-bit integers (we may take $\M(n)=O(n\log n)$ by~\cite{HarveyH21}).  Below, we analyze the complexity of Algorithm~\ref{alg:insependo}. 

\begin{prop}[GRH]\label{prop:specialendoruntime}

Algorithm~\ref{alg:insependo} is correct. Assuming GRH, for any   prime $\ell\in\{2,3,5\}$, integer $d\in \{1,2\}$, prime $p>4d$, and supersingular elliptic curve $E$ defined over $\FF_{p^2}$, Algorithm~\ref{alg:insependo} on input $(E,\ell,d)$  terminates in expected $O(p^{1/2}(\log p)^2(\llog p)^3)$ bit operations. 
\end{prop}
\begin{proof}
First, we argue that Algorithm~\ref{alg:insependo} is correct. Let $\alpha=\pi\widehat{\phi^{(p)}}\psi\phi$  with $\deg\phi=\ell^{k}$  be the output of Algorithm~\ref{alg:insependo}  on input $(E,\ell,d)$. We claim that $\alpha$ satisfies the hypotheses of Theorem~\ref{thm:genorder}. Because Algorithm~\ref{alg:insependo} uses non-backtracking walks, the $\ell^k$-isogeny $\phi$ is cyclic. Because the walk is truncated so that the  final vertex is the first curve in the walk with a $(d,\epsilon)$-structure, $\phi$ does not factor nontrivially through an isogeny to another curve with a $(d,\epsilon)$  structure. We conclude that $\alpha$ is an inseparable reflection.

We now bound the expected number of bit operations performed by the algorithm. 
    We can do Step~\ref{step:walklength} with Newton's method, for example, and the magnitude of the solution $t$ will be in $O(\log p)$. 
    We compute $\Phi_{\ell}$ and $\Phi_2$, if $d=2$, and store these polynomials. Since we treat $\ell$ and $d$ as constants, we ignore these costs, and in any case this computation can be done in $O(\ell^3\log^3\ell\llog\ell)$ expected time assuming GRH~\cite[Theorem 1]{BLS}. We can take one step in $G(p,\ell)$ using the modular polynomial $\Phi_{\ell}$. Let $E_0=E$ and $j_0=j(E_0)$. Suppose we are at vertex $j_i$. The neighbors of $j_i$ are the roots of $\Phi_{\ell}(j_i,Y)$. We can evaluate $\Phi_{\ell}(X,Y)$ at $(j_i,Y)$ in $O(\ell^2)=O(1)$ many multiplications and additions in $\FF_{p^2}$, the cost of which is dominated by the $O(\M(\ell \log p)(\llog p))$ bit operations needed to compute a random root of $\Phi_{\ell}(j_i,Y)$  using the randomized algorithm of~\cite{Rabin80}. To take a non-backtracking step, we compute a random root of $\Phi_{\ell}(j_i,Y)/(Y-j_{i-1})$ where $j_{i-1}$ is the previous vertex of the walk. Let $X$ denote the set of supersingular $j$-invariants in $\FF_{p^2}$ which are $d$-isogenous to their Galois conjugate. By Lemma~\ref{lem:check_deps}, we can test if $j_t$ is in $X$ by testing whether $\Phi_d(j_t,j_t^p)=0$ when $d>1$ and simply whether $j_t^p=j_t$ when $d=1$, both of which we can do with $O(\log p)$ multiplications in $\FF_{p^2}$.  Since the length of the walk is $O(\log p)$, and since we treat $\ell$ as a constant, Step~\ref{step:nbtwalk} takes $O(\M(\log p)(\log p)(\llog p))$ time.

    We now calculate the expected number of iterations of Step~\ref{step:nbtwalk}. Assuming GRH, by Proposition~\ref{prop:problowerbound}  a non-backtracking walk beginning at $E$ lands in $X$ with probability $\Omega\left(\frac{1}{\sqrt{p}\llog p}\right)$. Thus the expected number of non-backtracking walks we must take is $O(\sqrt{p}\llog p)$. Multiplying the expected number of walks by 
    the expected number of bit operations per walk and using $\M(n)=O(n\log n)$ by~\cite{HarveyH21} yields the cost 
    \[
    O(\M(\log p)(\log p)(\llog p) \cdot \sqrt{p}(\llog p)) = O(p^{1/2}(\log p)^2 (\llog p)^3).
    \]
    
    Let $j_0=j(E),j_1,\ldots,j_t$ be a sequence of adjacent $j$-invariants in $G(p,\ell)$ with $j_t\in X$. We next obtain the sequence of isogenies $\phi_i\colon E_i\to E_{i+1}$ for $i=0,\ldots,t-1$ and the $(d,\epsilon)$-structure $(E_t,\psi)$ with $O((\log p)^{O(1)})$ bit operations: for example, if $p$ is sufficiently large and if $j_{i+1}$ is a simple root of $\Phi_{\ell}(j_i,Y)$, given an equation for $E_i$ with $j$-invariant $j_i$, we can compute a short Weierstrass equation for an elliptic curve $E_{i+1}$ and a normalized $\ell$-isogeny $\phi_i\colon E_i\to E_{i+1}$ with $O(\ell^2)$ operations in $\FF_{p^2}$ using Elkies algorithm~\cite{Elk95} (see ~\cite[Algorithm 28]{GalPKC} for an explicit description of the algorithm). The time to compute the  sequence of isogenies associated to the path $j_0,\ldots,j_t$  is therefore dominated by the time required to complete the while-loop, since $\ell$ is a constant.  Similarly, the time required to truncate the path and to compute the $(d,\epsilon)$-structure is also dominated by the time required to complete the while-loop.
\end{proof}
\begin{rmk}
There are some natural optimizations which we do not explore here, such as testing more vertices along the path for the presence of $(d,\epsilon)$-structures,  or, more generally, testing whether a given curve $E_k$ along the path is $\ell'$-isogenous to a curve defined over $\FF_p$ with the algorithm of~\cite{CCS22}. 
\end{rmk}

\subsection{Computing a Bass suborder of \texorpdfstring{$\End(E)$}{EndE}}
In~\cite{EHLMP20}, the authors give a subexponential algorithm for computing a basis for  $\End(E)$ from a Bass suborder of $\End(E)$ but only give a heuristic algorithm for computing the Bass suborder.  
Theorems~\ref{thm:genorder} and~\ref{thm:bass_from_inseps} suggest the following approach to compute a Bass suborder of $\End(E)$: run Algorithm~\ref{alg:insependo} twice, first on the input $(E,3,2)$ and then on the input $(E,5,2)$, to produce two inseparable reflections 
$\alpha_i=\pi\widehat{\phi_i^{(p)}}\psi_i\phi_i$ of $E$ and the  Bass order $\Lambda_{\alpha_1\alpha_2}$ generated by $\alpha_1$ and $\alpha_2$.

\begin{rmk}
    The same heuristic the authors make in ~\cite{EHLMP20} to argue that the expectation of the number of calls to their algorithm for computing an endomorphism of a supersingular elliptic curve $E$ before producing a generating for a Bass order will be used in Section~\ref{sec:heuristic}  to argue that the expected number of calls to Algorithm~\ref{alg:insependo} before producing a generating set for $\ZZ+P$. This approach to computing a basis for $\End(E)$ is taken up in the next section.
\end{rmk}

\begin{algorithm}
\caption{Compute a Bass order contained in $\End(E)$}\label{alg:bass}
\KwIn{A supersingular elliptic curve $E/\FF_{p^2}$, two distinct primes $\ell_1,\ell_2\not=p$, and an integer $d$, with $d$ square-free, $d<p/4$, and $-dp\not\equiv 1 \pmod{4}$.}
\KwOut{A compact representation of a Bass order contained in $\End(E)$.}
 Use Algorithm~\ref{alg:insependo} twice, on input respectively $(E,\ell_1,d)$ and $(E,\ell_2,d)$, to compute two inseparable reflections $\alpha_1,\alpha_2$ of $E$\label{step:twoepsstructs}\;
\Return{$\Lambda=\langle 1,\alpha_1,\alpha_2,\alpha_1\alpha_2\rangle$}
\end{algorithm}

\begin{thm}[GRH]\label{thm:makebass}
Algorithm~\ref{alg:bass} is correct. Assuming GRH, if $p>8$ and $E$ is a supersingular elliptic curve over $\FF_{p^2}$ then on input $(E,3,5,2)$ Algorithm~\ref{alg:bass} terminates in expected $O(p^{1/2}(\log \log p)^2(\log\log p)^3)$ bit operations. 
\end{thm}

\begin{proof}
By Proposition~\ref{prop:specialendoruntime}, the two endomorphisms constructed in Step~\ref{step:twoepsstructs} are inseparable reflections. Write $\alpha_i=\pi\widehat{\phi_i^{(p)}}\psi_i\phi_i$ where $\phi_i\colon E\to E_i$ is an isogeny of degree $\ell_i^{k_i}$. Since $\ell_1\not=\ell_2$, the  kernels of $\phi_1$ and $\phi_2$ are distinct. Therefore Theorem~\ref{thm:genorder} implies $\Lambda$ is an order in $\End(E)$, and Theorem~\ref{thm:bass_from_inseps} implies $\Lambda$ is Bass. Thus, Algorithm~\ref{alg:bass} is correct. By Proposition~\ref{prop:specialendoruntime}, Step~\ref{step:twoepsstructs} terminates in expected $O(p^{1/2}(\log p)^2(\log \log p)^3)$ time. 
\end{proof}

\section{Computing \texorpdfstring{$\End(E)$}{End(E)} from a Bass suborder}\label{sec:provable}
In this section, we combine Algorithm~\ref{alg:bass}  and the algorithms in \cite{EHLMP20,Wes22} to produce an algorithm for computing a basis for the endomorphism ring of a supersingular elliptic curve $E$ over $\FF_{p^2}$ in expected time $O(p^{1/2}(\log p)^2(\log\log p)^3)$, conditional only on GRH. We begin with a high-level overview of the approach in \cite{EHLMP20} for computing a basis for $\End(E)$.  First, compute a Bass suborder $\Lambda\subseteq\End(E)$. This can be done with Algorithm~\ref{alg:bass} in expected $O(p^{1/2}(\log p)^2(\log\log p)^3)$ bit operations conditional on GRH by Theorem~\ref{thm:makebass}. Next,  enumerate maximal orders $\OO\subseteq \End^0(E)$ containing the Bass order $\Lambda$ until $\OO\cong \End(E)$ using  algorithms from~\cite{EHLMP20}.  We can efficiently check whether a given maximal order $\OO$ is isomorphic to $\End(E)$ using  the algorithms and reductions in \cite{EHLMP,Wes22}. This approach results in Algorithm~\ref{alg:endo_ring}.  

In Theorem~\ref{thm:EndE}, we show that Algorithm~\ref{alg:endo_ring} correctly computes $\End(E)$, and, conditional only on GRH, terminates in expected $O(p^{1/2}(\log p)^2(\log\log p)^3)$ bit operations.

\begin{rmk}
We ask  $\Lambda$ to be Bass because, under this assumption, we can prove that the number of maximal overorders of $\Lambda$ is subexponential in the size of $\Lambda$. 
    In general, there exists an order of size polynomial in $\log p$ 
    contained in $\End(E)$ which is contained in $\Omega(p)$  many distinct, 
    pairwise non-isomorphic maximal orders. For example, choose 
    $e=O(\log p)$ such that every supersingular $E'$ defined over 
    $\FF_{p^2}$ is connected by a $2^d$ isogeny to $E$ for some $d\leq 
    e$, and consider the order $\ZZ+2^e\End(E)$. We  claim  that this order has size polynomial in $\log p$ and embeds into the endomorphism ring of each supersingular elliptic curve over $\FF_{p^2}$, and therefore into a representative of each conjugacy class of maximal orders in $\End^0(E)$. For any supersingular 
    $E'$ there is an ideal $I\subseteq \End(E)$ such that $\OO_R(I)
    \cong \End(E')$ and $\Nrd(I)=2^d$ for some $d\leq e$. Then $\ZZ+2^e\End(E) \subseteq \ZZ+I \subseteq
    \OO_R(I)$. Since $E'$ was arbitrary, we conclude that $\ZZ+2^e
    \End(E)$ is contained in a representative of each conjugacy class 
    of maximal orders, and there are $\Omega(p)$ many conjugacy classes. Despite being contained in an exponentially large number of maximal orders, there is a basis for $\ZZ+2^e\End(E)$ whose size is polynomial in $\log p$ since $e=O(\log p)$ and $\End(E)$ has a basis whose size is polynomial in $\log p$.  
\end{rmk}

\begin{rmk}
    In~\cite{ES24}, Eisentr\"ager and Scullard give an algorithm for computing $\End(E)$ in polynomial time given a suborder $\Lambda$ of $\End(E)$ of polynomial size and a factorization of $\disc\Lambda$. Running Algorithm~\ref{alg:bass} to compute a Bass order $\Lambda=\Lambda_{\alpha_1\alpha_2}$, factoring $|\disc(\alpha_1\alpha_2/p)|$, and then using Algorithm 8.1 of~\cite{ES24} yields a faster algorithm for computing a basis for $\End(E)$ than the one outlined here, but with the same (conditional GRH) run time of $O(p^{1/2}(\log p)^2(\log\log p)^3)$. 
\end{rmk}

\begin{algorithm}
\caption{Compute End(E)}\label{alg:endo_ring}
\KwIn{A supersingular elliptic curve $E/\FF_{p^2}$, where $p>8$.}
\KwOut{A maximal order $\OO\subseteq B_{p,\infty}$ isomorphic to $\End(E)$.}

 Run Algorithm~\ref{alg:bass} on input $(E,3,5,2)$ 
  to compute a Bass order $\Lambda$ contained in $\End(E)$\label{step:makebassorder}\;
 Compute $a,b\in \QQ^{\times}$ and an isomorphism $f\colon\Lambda\otimes \QQ\to H(a,b)$\label{step:isotoBpinf}\;
Enumerate the maximal orders $\OO\supseteq f(\Lambda)$ 
until $\OO\cong \End(E)$\label{step:enummax}\;
\Return{$\OO$}
\end{algorithm}

We require an efficient algorithm for testing whether a maximal order $\OO$ is isomorphic to $\End(E)$. To do this, we will make use of an efficient algorithm for computing a supersingular curve $E'$ such that $\End(E')\cong \OO$. If $j(E')\in\{j(E),j(E)^p\}$, then $\OO\cong \End(E)$. An algorithm for producing a curve with endomorphism ring isomorphic to a given maximal order appears in~\cite{EHLMP}, which is efficient conditional on heuristics including GRH. The work of Wesolowski~\cite{Wes22} allows one to remove the heuristic assumptions (except for  GRH). We state that such an efficient algorithm exists here for completeness, noting that this algorithm  is due to results and algorithms in~\cite{Brok09,GPS,EHLMP,Wes22,CKMZ22}. See~\cite{EPSV} for a discussion of an efficient implementation of an algorithm addressing this problem. 
\begin{lem}\label{lem:deuring}
    There is an algorithm which, on input a basis for a maximal order $\OO$ of a quaternion algebra $B$ over $\QQ$ ramified at $p,\infty$, outputs a supersingular elliptic curve $E$ defined over $\FF_{p^2}$ such that $\End(E)\cong \OO$. Conditional on GRH, the algorithm runs in time polynomial in the size of $\OO$.
\end{lem}
\begin{proof}
    This follows by combining algorithms and results in~\cite{Piz80, Brok09, KLPT, GPS, EHLMP,CKMZ22, Wes22}.  First, we compute a quaternion algebra $B_{p,\infty}$ ramified at $p$ and $\infty$, a supersingular elliptic curve $E_0$, and a maximal order $\OO_0$ in $B_{p,\infty}$ such that $\OO_0\cong \End(E_0)$ under some explicit isomorphism. This can be done in time polynomial in $\log p$, conditional on GRH~\cite[Proposition 3]{EHLMP}. Next, we compute an isomorphism $f\colon \OO\otimes\QQ\to B_{p,\infty}$ of quaternion algebras, which can be done  in time polynomial in $\log p$~\cite[Proposition 4.1]{CKMZ22}. 

    We now compute a supersingular elliptic curve $E$ such that $\End(E)\cong \OO$. First, we compute a connecting ideal $J$ between $\OO_0$ and $f(\OO)$. By Theorem 6.4 of~\cite{Wes22}, assuming GRH, we can, in expected polynomial time, compute an equivalent ideal $I$ to $J$ such that the norm of $I$ is $B$-powersmooth for some $B=O((\log p)^c)$ for some constant $c$, meaning that if $p^e$ exactly divides $\Nrd(I)$ then $p^e\leq B$. Since the norm of $I$ is $B$-powersmooth, we can efficiently compute the corresponding isogeny $\phi_I\colon E_0\to E$~\cite[Proposition 4]{EHLMP}. The codomain $E$ of $\phi_I$ is a curve whose endomorphism ring is isomorphic to $\OO$, since 
    \[
\End(E)\cong \OO_R(I)\cong \OO_R(J) = f(\OO)\cong \OO. 
\]
Thus the algorithm is correct. 
\end{proof}

We require a bound on the number of maximal orders containing a given Bass order $\Lambda$. Below we show that that we may bound this quantity in terms of the number of divisors of the reduced discriminant of the order. Assuming $\Lambda$ has size polynomial in $\log p$, this implies that the number of maximal orders containing $\Lambda$ grows at most subexponentially in $\log p$. 

\begin{lem}\label{lem:bass_maxorder_bound}
    Let $\Lambda\subseteq B_{p,\infty}$ be a Bass order. The number of 
    maximal orders in $B_{p,\infty}$ containing $\Lambda$ is $O((\discrd(\Lambda))^{\epsilon})$, for every $\epsilon >0$. 
\end{lem}

\begin{proof}
    By Proposition 4.2 of~\cite{EHLMP20} and the local-global dictionary for orders~\cite[Theorem 9.1.1]{V2013}, the number of maximal overorders of $\Lambda$ is bounded by 
\[
\prod_{\substack{q \text{ prime} \\ q\neq p}} v_q(\discrd(\Lambda))+1,
\]
and this quantity is equal to the number of divisors of $\discrd(\Lambda)/p^{v_p(\discrd(\Lambda))}$. The number of divisors of an integer $n$ is $O(n^{\epsilon})$ for every $\epsilon>0$~\cite[Theorem 315]{HW08}. The claim of the lemma follows. 

\end{proof}
We now prove the main theorem of this section, which states that our algorithm computes the endomorphism ring of a supersingular elliptic curve $E$ defined over $\FF_{p^2}$ in $O(p^{1/2}(\log p)^2(\log \log p)^3)$ bit operations, conditional on GRH (and assuming no further heuristics).
\begin{thm}[GRH]\label{thm:EndE}
  Algorithm~\ref{alg:endo_ring} is correct. Assuming GRH, Algorithm~\ref{alg:endo_ring} terminates in expected \[O(p^{1/2}(\log p)^2(\log \log p)^3)\] bit operations. 
\end{thm}
\begin{proof}
By Theorem~\ref{thm:makebass}, Step~\ref{step:makebassorder} runs in expected time $O(p^{1/2}(\log p)^2(\log \log p)^3)$. Moreover, $\Lambda$ is Bass.  We now discuss Step~\ref{step:isotoBpinf}. First, compute the Gram matrix $G$ under the trace pairing of the basis $1,\alpha_1,\alpha_2,\alpha_1\alpha_2$ for $\Lambda$: by the discussion in Section 4, we need to compute a single trace, namely $\Trd(-\alpha_1\alpha_2/p)$. This trace can be computed in time polynomial in $\log p$ with a generalization of Schoof's algorithm \cite{Koh96, BCEMP}, since $\rho\coloneqq-\alpha_1\alpha_2/p$ is a cyclic isogeny of degree $2^23^{2k_1}5^{2k_2}$  and $k_1,k_2=O(\log p)$. With $G$, compute $a,b\in \QQ^{\times}$ such that $\Lambda\otimes \QQ$ is isomorphic to $H(a,b)$ with the Gram-Schmidt process. 

We now outline how to do Step~\ref{step:enummax}.
 We factor $\disc(\rho)$ to obtain a factorization of $\discrd(\Lambda)=p^4|\disc(\rho)|$. Since $\rho$ is the product of $2+2k_1+2k_2=O(\log p)$ isogenies of degree at most $5$, the degree of $\rho$ is $O(p^C)$ for some $C$. This implies $-\disc(\rho) = O(p^C)$ as well.
Therefore we can factor $\disc(\rho)$ in time subexponential in $\log p$~\cite[Theorem 1]{LP92}. 
  For each $q|\discrd(\Lambda)$ such that $q\not=p$, we can enumerate maximal $\ZZ_q$-orders containing $f(\Lambda)\otimes \ZZ_q$ efficiently using Algorithm 4.3 of~\cite{EHLMP20} and then enumerate the $\ZZ$-orders containing $f(\Lambda)$; see Steps 1(a) and 3(a) in Algorithm 5.4 of~\cite{EHLMP20}. For each maximal order $\OO$ containing $f(\Lambda)$, compute an elliptic curve $E'$ with $\End(E')\cong \OO$. This can be done in polynomial time in $\log p$ by Lemma~\ref{lem:deuring}. 
If $j(E')\in \{j(E),j(E)^p\}$, we return $\OO$. Thus the algorithm is correct. 

By Lemma~\ref{lem:bass_maxorder_bound}, the number of maximal overorders of $f(\Lambda)$ is $O(p^{\epsilon})$ for every $\epsilon>0$.  We conclude that Step~\ref{step:enummax} takes $O(p^{\epsilon})$ time for any $\epsilon>0$. In particular, the expected time required to complete Step~\ref{step:enummax} is dominated by the expected time required to complete Step~\ref{step:makebassorder}. 
\end{proof}

\section{The number of inseparable reflections needed to generate \texorpdfstring{$\End(E)$}{End(E)}}\label{sec:inseparablegens}

Let $E$ be a supersingular elliptic curve $E$ over $\FF_{p^2}$. 
Let $P\coloneqq \pi\Hom(E,E^{(p)})\subseteq \End(E)$ be the ideal of inseparable endomorphisms of $E$. 
Theorem~\ref{thm:makebass} implies that for an appropriate choice of parameters, with two calls to Algorithm~\ref{alg:insependo} we compute a generating set for a Bass order $\Lambda$ contained in $\ZZ+P$. If one had a basis for $\ZZ+P$, rather than just a Bass suborder, then the algorithms of Voight~\cite{V2013} can compute a basis for $\End(E)$ efficiently, since $\End(E)$ is the unique maximal order containing $\ZZ+P$ by Proposition~\ref{prop:z+p}. This raises the following question: how many calls to Algorithm~\ref{alg:insependo} does one expect to make before computing a generating set for $\ZZ+P$? 
In this section, we give a heuristic argument   showing that the number of inseparable endomorphisms required to generate $\ZZ+P$ is bounded above by a constant, independent of the field; empirically, four inseparable reflections do the trick more often than not. This results in a second algorithm  for computing $\End(E)$:  compute inseparable reflections with Algorithm~\ref{alg:insependo} until finding a generating set for $\ZZ+P$, and then compute a basis for the maximal order containing $\ZZ+P$ using the algorithms of~\cite{V2013}. This algorithm will be  slower than Algorithm~\ref{alg:endo_ring}, but is simpler to implement; we discuss further implementation details in the Appendix.

\subsection{The expected number of inseparable reflections in a generating set for \texorpdfstring{$\ZZ+P$}{Z+P}}\label{sec:heuristic}

Suppose we run Algorithm~\ref{alg:bass} twice on input  $E$, a supersingular elliptic curve defined over $\FF_{p^2}$, producing two   orders $\Lambda_{\alpha_1\alpha_2}$ and $\Lambda_{\alpha_3\alpha_4}$  in $\ZZ+P\subseteq \End(E)$ spanned respectively by  $1,\alpha_1,\alpha_2,\alpha_1\alpha_2$ and $1,\alpha_3,\alpha_4,\alpha_3\alpha_4$, where $\alpha_i$ is an inseparable reflection for every $i=1,\ldots,4$. Let $\Lambda$ be the order in $\End(E)$ generated by $\alpha_1,\alpha_2,\alpha_3,\alpha_4$.

\begin{center}
 \begin{tikzcd}
  &  \End(E) \arrow[d, dash, "p"]&   \\
 &  \mathbb Z+P \arrow[d, dash]&   \\
&  \Lambda\arrow[dr,dash] &   \\
\Lambda_{\alpha_1\alpha_2}  \arrow[ur,dash]& & \Lambda_{\alpha_3\alpha_4} \\
\end{tikzcd}
\end{center}
Then 
\[\discrd(\Lambda)=\discrd( \mathbb Z+P)\cdot[\mathbb Z+P:\Lambda]=p^2\cdot[\mathbb Z+P:\Lambda],
\] and $\discrd(\Lambda)$ divides both $\discrd(\Lambda_{\alpha_1\alpha_2})$  and $\discrd(\Lambda_{\alpha_3\alpha_4})$. Defining $\rho_1=\frac{\alpha_1\alpha_2}{p}$ and $\rho_2=\frac{\alpha_3\alpha_4}{p}$,  we have  
\[
\discrd\Lambda_{\alpha_1\alpha_2} =p^2|\disc(\rho_2)| \text{ and }  \discrd\Lambda_{\alpha_3\alpha_4} =p^2|\disc(\rho_1)|. 
\]
 In particular, $[\mathbb Z+P:\Lambda]=\frac{\discrd(\Lambda)}{p^2}$ divides $\gcd(\disc(\rho_1),\disc(\rho_2))$. If $\gcd(\disc\rho_1,\disc\rho_2)=1$, then $\Lambda = \ZZ+P$ and the four inseparable endomorphisms $\alpha_1,\alpha_2,\alpha_3,\alpha_4$ generate $\ZZ+P$. If  the distribution of the integers $\disc(\rho_1)$ and $\disc(\rho_2)$ follow the same distribution as two random integers, then $\disc(\rho_1)$ and $\disc(\rho_2)$ are coprime with probability $6/\pi^2$. Assuming this,  four calls to Algorithm~\ref{alg:insependo} produce a generating set for $\ZZ+P$ with  at least $6/\pi^2\approx 0.6$ probability.

Unfortunately, the integers $D_i=\disc(\rho_i)$ are not distributed like random integers. First of all, the integer $D_i$ is a discriminant, which imposes congruency conditions on $D_i$. Second, the prime $p$ is not split in $\ZZ[\rho_i]$, imposing another congruence condition. Finally, the fact that $\rho_i$ is an endomorphism of smooth degree enforces relations in the ideal class group of $\ZZ[\rho_i]$. In any case, the following heuristic suffices for our purposes:

\begin{heuristic}\label{heuristic:coprime}
Let $p>3$ be a prime and let $\ell<p/4$ be a prime. Let $E$ be a supersingular elliptic curve over $\FF_{p^2}$. There exists a constant $c>0$, independent of $p$, such that if $\rho_i=-\alpha_{i1}\alpha_{i2}/p$ where $\alpha_{ij}$, $1\leq i,j\leq 2$ are the outputs of four calls to Algorithm~\ref{alg:insependo} on input $(E/\FF_{p^2},\ell,1)$, then $\Pr[\gcd(\disc\rho_1,\disc\rho_2)=1]\geq c$.
\end{heuristic}

The following theorem follows from the above discussion:
\begin{thm}\label{thm:boundedgens}
Assume Heuristic~\ref{heuristic:coprime}. Let $p>3$ and $\ell<p/4$ be primes, and let $E$ be a supersingular elliptic curve defined over $\FF_{p^2}$. Then the expected number of calls to Algorithm~\ref{alg:insependo} on input $(E,\ell,1)$ in order to produce a generating set for $\ZZ+P$ is bounded from above by a constant, independent of $p$. 
\end{thm}
Heuristic \ref{heuristic:coprime} is~\cite[Heuristic 5.2]{EHLMP20} which is assumed in~\cite[Theorem 5.3]{EHLMP} to prove that~\cite[Algorithm 5.1]{EHLMP20} produces a Bass order in $\End(E)$ and terminates in expected $O(p^{1/2+\epsilon})$ time. We use the heuristic in a new way. 

\begin{rmk}\label{rmk:secondheuristic}
This heuristic argument applies to any pair of orders generated by two pairs of of two non-commuting elements of a maximal order in  $B_{p,\infty}$, the quaternion algebra ramified at $p$ and $\infty$. 
Let $\alpha_1,\alpha_2$ be two arbitrary non-commuting elements of a quaternion order $\OO\subseteq B_{p,\infty}$ and let 
$\Lambda=\langle \alpha_1,\alpha_2\rangle\coloneqq \ZZ+\ZZ\alpha_1+\ZZ\alpha_2+\ZZ\alpha_1\alpha_2$ be the order they generate, and let $T_i:=\Trd(\alpha_i)$, $N_i:=\Nrd(\alpha_i)$ for $i=1,2$, and let $T_{12}=\Trd(\alpha_1\alpha_2)$. Then the discriminant of 
$\Lambda$ is 
\[
\det\begin{pmatrix}
2 & T_1 & T_2 & T_{12} \\ 
T_1 & 2N_1 & T_1T_2-T_{12} & N_1T_2 \\ 
T_2  & T_1T_2-T_{12} & 2N_2 & N_2T_1 \\ 
T_{12} & T_2N_1 & T_1N_2 & 2N_1N_2
\end{pmatrix}
= \left(\frac{1}{4}\disc(T_2\alpha_1+T_1\alpha_2-2\alpha_1\alpha_2)\right)^2~.
\]

 Suppose now that we sample $\alpha_{11},\alpha_{12},\alpha_{21},\alpha_{22}$ uniformly (in some reasonable sense - for example, uniformly from the elements of $\OO$ whose norm is bounded by some fixed polynomial in $p$) and that $\alpha_{i1}\alpha_{i2}\not=\alpha_{i2}\alpha_{i1}$ for $i=1,2$. 
  Define $\Lambda_{1}$ and $\Lambda_2$ to be the orders generated by $\alpha_{11},\alpha_{12}$ and $\alpha_{21},\alpha_{22}$, respectively, let 
 \[
 \rho_i \coloneqq(\Trd\alpha_{i2})\alpha_{i1}+(\Trd\alpha_{i1})\alpha_{i2}-2\alpha_{i1}\alpha_{i2},
 \]
 let $D_i=\disc \rho_i,$, and let $\Lambda$ denote the order generated by the $\alpha_{ij}$ for $1\leq i,j\leq 2$. Then 
 \[
 [\OO:\Lambda] = \frac{\discrd\Lambda}{\discrd\OO}, \quad [\OO:\Lambda_i]=\frac{\discrd\Lambda_i}{\discrd\OO}=\frac{|D_i|}{4\discrd\OO}
 \]
 and $[\OO:\Lambda]$ divides 
 \[
 \gcd([\OO:\Lambda_1],[\OO:\Lambda_2])=\gcd\left(\frac{D_1}{4\discrd\OO},\frac{D_2}{4\discrd\OO}\right).
 \]
Therefore the $\alpha_{ij}$ will generate $\OO$ with probability at least the probability that 
\[
\gcd(D_1,D_2)=4\discrd\OO.
\]
A reasonable heuristic would then be that this probability is bounded from below by a constant, independent of $p$. We will explore this experimentally in the next section as well. 
\end{rmk}

\subsection{Computational experiments}\label{subsec:evidence}
We implemented Algorithm~\ref{alg:insependo} along with the various algorithms discussed in this section in order to empirically determine the expected value of the number of inseparable reflections of $E$ required before generating $\End(E)$. We believe this expectation is bounded by a constant, independent of $p$ or $E$. We restricted our attention to elliptic curves $E$ defined over $\mathbb{F}_{p^2}$ but not over $\FF_p$, since there are asymptotically faster algorithms for computing the endomorphism ring of such curves. 

To experimentally test Heuristic~\ref{heuristic:coprime} and to understand the expected number of inseparable reflections in a generating set for $\ZZ+P$, we conducted the following experiment. For $n\in \{16,17,\ldots,32\}$, we repeated the following procedure 100 times: we chose the first prime $p$ after $2^n$ and computed a pseudorandom supersingular $j$-invariant in $\FF_{p^2}-\FF_p$ by taking a random walk in $G(p,2)$ of length $\lfloor\log_2p\rfloor$. We then generated four inseparable reflections $\alpha_i$, $1\leq i\leq 4$, of degree $2^{2t}p$  for $i=1,3$ and $3^{2s}p$  for $i=2,4$. Next, we tested whether $\gcd\left(\disc\left(\frac{\alpha_1\alpha_2}{p}\right),\disc\left(\frac{\alpha_3\alpha_4}{p}\right)\right)=1$ and whether $ 1,\alpha_1,\alpha_2,\alpha_3,\alpha_4$ generate $\ZZ+P$. We report the sample mean for the random variable which is $1$ when $\gcd\left(\disc\left(\frac{\alpha_1\alpha_2}{p}\right),\disc\left(\frac{\alpha_3\alpha_4}{p}\right)\right)=1$ and $0$ otherwise, and  the sample mean of the  random variable which is $1$ when $ 1,\alpha_1,\alpha_2,\alpha_3,\alpha_4$ generate $\mathbb Z+P$ and $0$ otherwise in Table \ref{fig:table-probability1}.
\begin{figure}
\includegraphics[scale=.75]{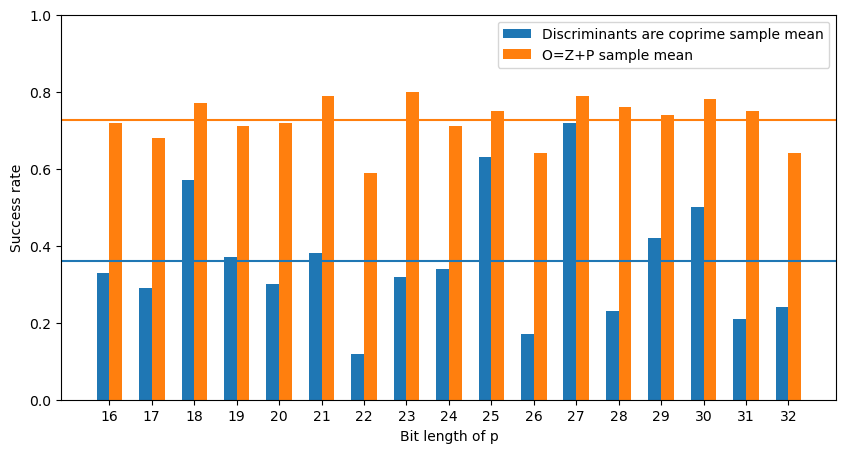}
\caption{Collected data for testing Heuristic~\ref{heuristic:coprime}. Orange bars represent the experimental probability that $\gcd\left(\disc\left(\frac{\alpha_1\alpha_2}{p}\right),\disc\left(\frac{\alpha_3\alpha_4}{p}\right)\right)=1$, blue bars represent the experimental probability that  $ 1,\alpha_1,\alpha_2,\alpha_3,\alpha_4$ generate $\mathbb Z+P$, where $\alpha_i$ are inseparable reflections of a supersingular elliptic curve. Averages of the two frequencies are plotted as well.}
\label{fig:table-probability1}
\end{figure}
\begin{figure}
    \centering
    \includegraphics[scale=.75]{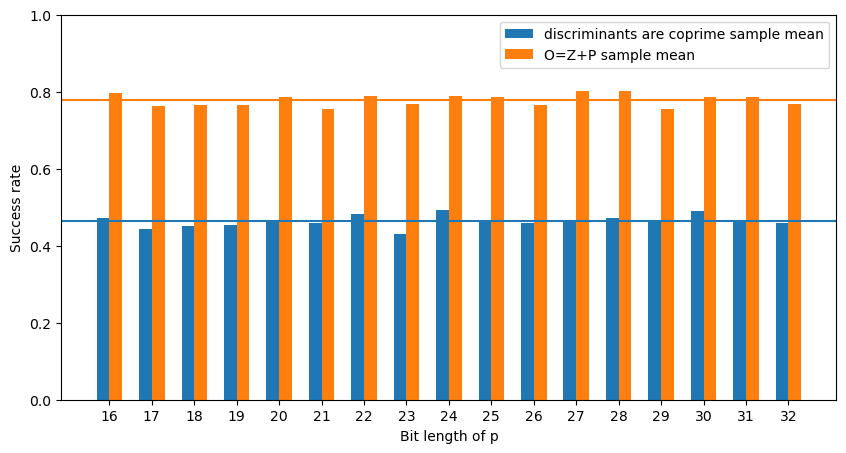}
    \caption{Collected data for testing heuristic in Remark~\ref{rmk:secondheuristic}. Orange bars represent the experimental probability that $\gcd\left(\disc\left(D_1,D_2\right)\right)=4p^2$, blue bars represent the experimental probability that  $1,\alpha_{11},\alpha_{12},\alpha_{21},\alpha_{22}$ generate $\mathbb Z+P$ where $\alpha_{ij}$ are random elements of $\ZZ+P$ in a random maximal order in $B_{p,\infty}$ and $D_i$ is the discriminant of $\rho_i \coloneqq(\Trd\alpha_{i2})\alpha_{i1}+(\Trd\alpha_{i1})\alpha_{i2}-2\alpha_{i1}\alpha_{i2}$. Averages of the two frequencies are plotted as well.}
    \label{fig:table-probability2}
\end{figure}

The data does not seem to invalidate Heuristic~\ref{heuristic:coprime}, but it also does not illuminate what the actual probability is that two inseparable reflections have coprime discriminants. The data does support our desired conclusion, namely that on average, the number of inseparable reflections needed to generate $\ZZ+P$ is bounded from above by a constant, independent of $p$. In particular, that constant appears to be bounded from above by four! In any case, the coprimality of the discriminants is not a necessary condition for the inseparable reflections to generate $\ZZ+P$. 

An idealized version of the algorithm sketched in Section~\ref{sec:inseparablegens} would generate random  endomorphisms in $\ZZ+P$ of bounded norm, rather than ``structured'' endomorphisms such as the inseparable reflections output by Algorithm~\ref{alg:insependo}. One might wonder whether the heuristic suggested in~\ref{rmk:secondheuristic} holds, and how many random elements of $\ZZ+P$ are required to generate $\ZZ+P$. We conducted the following numerical experiment. For $n\in \{16,17,\ldots,32\}$, we repeated the following procedure 100 times: we chose the first prime $p$ after $2^n$ and computed a pseudorandom maximal order $\OO$ in the quaternion algebra $B_{p,\infty}$ ramified at $p$ and $\infty$.  We then computed $\ZZ+P$, where $P$ is the unique $2$-sided ideal of $\OO$ of reduced norm $p$ and sampled four random elements $\alpha_{11},\alpha_{12},\alpha_{21},\alpha_{22}\in \ZZ+P$. We compute 
\[
 \rho_i \coloneqq(\Trd\alpha_{i2})\alpha_{i1}+(\Trd\alpha_{i1})\alpha_{i2}-2\alpha_{i1}\alpha_{i2}
 \]
 and $D_i=\disc\rho_i$ and then tested whether $\gcd\left(D_1,D_2\right)=4\discrd(\ZZ+P)=4p^2$ and whether $ 1,\alpha_1,\alpha_2,\alpha_3,\alpha_4$ generate $\ZZ+P$. The sample means are reported in Table~\ref{fig:table-probability2}. The probabilities that $\gcd(D_1,D_2)=4p^2$ and that $\{1,\alpha_{11},\alpha_{12},\alpha_{21},\alpha_{22}\}$ generate $\ZZ+P$ do not appear to decay as $p$ increases in the range $[2^{16},2^{32}]$.

\appendix

\section{Appendix}\label{sec:appendix}
In this appendix, we  discuss an algorithm following the idea suggested in Section~\ref{sec:inseparablegens}, that is, to compute a basis for the endomorphism ring of a supersingular elliptic curve $E$ by making repeated calls to Algorithm~\ref{alg:insependo} to produce inseparable endomorphisms.  The extra ingredients include a generalization of Schoof's algorithm for computing the trace of an endomorphism, some algorithms of Voight~\cite{V2013} for local and global quaternion orders, and linear algebra. Below, we provide the details regarding the linear algebra necessary to complete the algorithm. Our implementation in SageMath is available at~\url{https://github.com/travismo/inseparables}.

The algorithm goes as follows: we first compute three inseparable reflections $\gamma_1,\gamma_2,\gamma_3$ of $E$. Let $P\coloneqq \Hom(E^{(p)},E)\pi$ be the ideal of inseparble endomorhisms of $E$; then $P$ is the unique $2$-sided ideal of reduced norm $p$ in $\End(E)$. Defining $\gamma_0=1$,  we next compute the Gram matrix $G=(\Trd(\gamma_i\widehat{\gamma_j}))$ for the sequence $\Gamma=(\gamma_0,\gamma_1,\gamma_2,\gamma_3)$; this is where we require a generalization of Schoof's algorithm~\cite{BCEMP}. Then $\Gamma$ is a basis for $\End^0(E)$ as a $\QQ$-vector space if and only if $\det(G)\not=0$, which we now assume. At this point, we have computed $\End^0(E)$ as a quadratic module: if we let $Q(x)=x^TGx$ denote the quadratic form induced by $G$ on $\QQ^4$, then $(\End^0(E),\deg)\cong (\QQ^4,Q)$. Having computed $\End^0(E)$ as a quadratic module, we determine its structure  as a quaternion algebra: we compute a multiplication table for the basis $\Gamma$. We then compute the order $\OO\subseteq \End(E)$ generated by $\gamma_{0},\gamma_1,\gamma_2,\gamma_3$. Finally, we enlarge the order $\OO$ by computing additional inseperable reflections until $\OO=\ZZ+P$. As mentioned above, a basis for $\End(E)$ is efficiently recovered from a basis for $\ZZ+P$ using algorithms of Voight~\cite{Voight}. 
\subsection{Computing a quadratic submodule of \texorpdfstring{$\ZZ+P$}{Z+P}} 
Recall that the output of Algorithm~\ref{alg:insependo} on input $E$ is  a trace-zero endomorphism of $E$ belonging to $P$.
We  assume that, by running Algorithm~\ref{alg:insependo} three times (with $d=1$ for simplicity) we have computed three inseparable reflections $\gamma_1,\gamma_2,\gamma_3$ of $E$ and we define $\gamma_0=1\in \End(E)$.
Since $d=1$, for $i=1,2,3$ we have
$$\gamma_i = \pi_p\widehat{\phi_i^{(p)}} \phi_i,$$
where $\phi_i\colon E \to E_i$ is a separable isogeny.

Let  $\Lambda:= \ZZ\gamma_0 + \ZZ\gamma_1 + \ZZ\gamma_2 + \ZZ\gamma_3$ be the $\ZZ$-span of  $\gamma_0,\gamma_1,\gamma_2$ and $\gamma_3$. 
Let $G\coloneqq (\Trd(\gamma_i\widehat{\gamma_j}))$ be the Gram matrix for  $\gamma_0,\gamma_1,\gamma_2,\gamma_3$. 
First, by Proposition~\ref{prop:trace_zero_insep}, we have 
$\Trd(\gamma_i)=0$ for $i=1,2,3$. For $1\leq i<j\leq 3$ define 
\[\rho_{ij} \coloneqq \widehat{\phi_i} \phi_i^{(p)}\widehat{\phi_j^{(p)}} \phi_j.
\]
Then $\Trd(\gamma_i\widehat{\gamma_j}) = p \Trd(\rho_{ij})$ and the Gram matrix of the basis $\gamma_0,\gamma_1,\gamma_2,\gamma_3$ is 
\[
G\coloneqq (\Trd(\gamma_i\widehat{\gamma_j}))_{0\leq i,j\leq 3}=\begin{pmatrix}
2 & 0 & 0 & 0 \\
0 & 2p\deg(\phi_1)^2 & p \Trd(\rho_{12}) & p \Trd(\rho_{13}) \\ 
0 & p \Trd(\rho_{13}) & 2p\deg(\phi_2)^2 & p \Trd(\rho_{23}) \\ 
0& p \Trd(\rho_{13}) & p \Trd(\rho_{23}) & 2p\deg(\phi_3)^2
\end{pmatrix}.
\]
We can compute the entries of $G$ with an algorithm for computing the trace of an endomorphism represented as a sequence of low degree isogenies with a generalization of Schoof's algorithm~\cite{BCEMP}. If $\det(G)\not=0$, then $\Lambda$ is a lattice in $\ZZ+P\subseteq \End(E)$, which we now assume. 
Therefore, as quadratic $\ZZ$-modules, we have $(\Lambda,\deg) \cong (\ZZ^4,G)$
under the isomorphism which sends $\gamma_i$ to the $i$th standard basis vector in $\ZZ^4$.
\subsection{From a quadratic module to an order in a quaternion algebra}

With the Gram matrix $G$ of the basis $\gamma_0,\gamma_1,\gamma_2,\gamma_3$ in hand, we move on to determining the structure of $\Lambda\otimes\QQ$ as a quaternion algebra. We discuss two approaches: the first involves computing an embedding of $\Lambda = \ZZ\gamma_0 + \ZZ\gamma_1 + \ZZ\gamma_2 + \ZZ\gamma_3$ into a quaternion algebra $H(a,b)$ such that $(\ZZ^4,G)\cong (H(a,b),\Nrd)$ are isomorphic as quadratic spaces. A second approach is to directly compute a multiplication table for $\gamma_1,\gamma_2,\gamma_3$, i.e. 
computing rational numbers $m_{rst}$ such that 
$\gamma_r\gamma_s = \sum_t m_{rst} \gamma_t$.
We discuss both in detail. In the first, we compute the $LDL^T$-decomposition of $G$ and read off $a$ and $b$ from the second and third entries of $D$. In the other, we solve for $m_{rst}$ by setting up a system of equations using $G$. 

\subsubsection{Computing an isomorphism of quaternion algebras 
using the Gram--Schmidt process}\label{sec:GS}

Let $G\coloneqq (G_{rs})_{0\leq r,s\leq 3} \coloneqq (\Trd(\gamma_r\widehat{\gamma_s}))_{0\leq r,s\leq 3}$ be the Gram matrix for the basis $\{\gamma_0,\gamma_1,\gamma_2,\gamma_3\}$ of a lattice $\Lambda$ in $\End(E)$. 
One approach to giving $\Lambda\otimes \QQ$ the structure of a quaternion algebra is as follows. First, 
we diagonalize the quadratic form induced by $G$ (to be precise, we compute the $LDL^T$-decomposition of 
$G$). We obtain a lower-triangular matrix $L$ with $1$'s on the diagonal and a diagonal matrix $D$ such 
that $G=LDL^T$. Denote the diagonal entries of $D$ by $d_0=2,d_1,d_2,d_3$ and define $a,b\in \QQ$ by $d_1=-2a$, 
$d_2=-2b$. Define $H(a,b)$ to be the quaternion algebra with basis $1,i,j,ij$ such that  $i^2=a$, $j^2=b$, and $ij=-ji$. Define  $R=L^T$ and $\tilde{\gamma_i}=\sum_{j} (R^{-1})_{ij}\gamma_{j}$ for $i=0,1,2,3$. Then $\{\tilde{\gamma_0},\tilde{\gamma_1},\tilde{\gamma_2},\tilde{\gamma_3}\}$  is the result of the application of the Gram--Schmidt process to the basis $\{\gamma_0,\gamma_1,\gamma_2,\gamma_3\}$ of $\Lambda\otimes \QQ$.  Since $\Trd(\tilde{\gamma_1})=0$, we have 
\[
\tilde{\gamma_1}^2 = -\tilde{\gamma_1}\widehat{\tilde{\gamma_1}} = \frac{-1}{2}\Trd(\tilde{\gamma_1}\widehat{\tilde{\gamma_1}}))=\frac{-d_1}{2} = a.
\]
Similarly,  
$(\tilde{\gamma_2})^2=b$. Since $\tilde{\gamma_3}$ and $\tilde{\gamma_1}\tilde{\gamma_2}$ are both 
orthogonal to each of $1,\tilde{\gamma_1},\tilde{\gamma_2}$, there exists $c\in \QQ$ such that  
$\tilde{\gamma_3}=c\tilde{\gamma_1}\tilde{\gamma_2}$. Taking reduced norms, we obtain 
\[
\frac{d_3}{2}=\Nrd(\tilde{\gamma_3})=\Nrd(c\tilde{\gamma_1}\tilde{\gamma_2})=\frac{c^2d_1d_2}{4}.
\]
Define $c'\coloneqq \sqrt{\frac{2d_3}{d_1d_2}}$. We therefore obtain an  isomorphism of quadratic spaces 
\begin{align*}
(\End^0(E),\deg)&\to (H(a,b),\Nrd) \\ 
x_0+x_1\tilde{\gamma_1}+x_2\tilde{\gamma_2}+x_3\tilde{\gamma_3}&\mapsto  x_0+x_1i+x_2j+x_3c'ij.
\end{align*}
This map factors through the map of quadratic modules $f\colon \End^0(E)\to (\QQ^4,G)$ which sends $\gamma_r$ to $e_r$, the $r$-th standard basis vector of $\QQ^4$, via the map 
$g\colon (\QQ^4,G)\to (H(a,b),\Nrd)$ obtained from sending the rows of $(L^T)^{-1}$ to the basis $1,i,j,ij$. 
The isomorphism $(\End^0(E),\deg)\cong (H(a,b),\Nrd)$  induces an isomorphism of quaternion algebras between $\End^{0}(E)$ with either $H(a,b)$ (in the case that $c=c'$) or $H(a,b)^{\text{op}}$ (in the case that $c=-c'$). 

\subsubsection{Computing a multiplication table using linear algebra}\label{sec:multtable}
An alternative method for representing $\End^0(E)$ using the basis $\{\gamma_0,\gamma_1,\gamma_2,\gamma_3\}$ is to compute a multiplication table, i.e. rational numbers $m_{rst}$, for $0\leq r,s,t\leq 3$, such that 
\[
\gamma_r\gamma_s = \sum_{t=0}^3 m_{rst} \gamma_t.\]
We sketch this approach, although we do not use it in our implementation. 

To compute the multiplication table $\{m_{rst}\}$, we use $G=(G_{rs})_{0\leq r,s\leq 3}$ and linear algebra. In particular, we use $G$ to set up a system of linear equations whose solutions are the $\{m_{rst}\}$ we seek. 

\begin{prop}\label{prop:mult_table}
Let $\gamma_0=1,\gamma_1,\gamma_2,\gamma_3$ be as above.  
Define 

\[ m_{000}=1,\quad   m_{rrt}= \begin{cases} -\frac{G_{rr}}{2} &: t=0 \text{ and } 1\leq r\leq 3\\ 
0 &: 0\leq r\leq 3 \text{ and } 1\leq t\leq 3.
\end{cases}
\]
Let  $\{m_{120},\ldots,m_{123}\}$, $\{m_{130},\ldots,m_{133}\}$, and $\{m_{230},\ldots,m_{233}\}$ respectively solve the following three systems of linear equations:
\[
G\begin{pmatrix}
    x_0 \\ x_1 \\ x_2 \\ x_3
\end{pmatrix} = \begin{pmatrix}
    -G_{12} \\ \frac{1}{2} G_{11}G_{21} \\ \frac{1}{2}G_{22}G_{11} 
    \\ 2\Trd(\gamma_1\gamma_2\widehat{\gamma_3})
\end{pmatrix}, 
\quad 
G\begin{pmatrix}
    x_0 \\ x_1 \\ x_2 \\ x_3
\end{pmatrix} = \begin{pmatrix}
    -G_{13} \\ \frac{1}{2} G_{11}G_{31}  
    \\ 2\Trd(\gamma_1\gamma_2\widehat{\gamma_3}) \\ \frac{1}{2}G_{33}G_{11}
\end{pmatrix},
\quad 
G\begin{pmatrix}
    x_0 \\ x_1 \\ x_2 \\ x_3
\end{pmatrix} \begin{pmatrix}
    -G_{23} \\ 2\Trd(\gamma_1\gamma_2\widehat{\gamma_3}) \\
\frac{1}{2} G_{22}G_{31} \\ 
 \frac{1}{2}G_{33}G_{21}
\end{pmatrix}.
\]
Finally, for $0\leq s < r \leq 3$, let 
\[
m_{srt}=\begin{cases}
\Trd(\gamma_r\gamma_s)-m_{rs0}
&: t=0 \\
-m_{rst} &: 1\leq t \leq 3.
\end{cases}
\]
Then $
(2\Trd(\gamma_1\gamma_2\widehat{\gamma_3}))^2 = \det G,$
and for $0\leq r \leq s \leq 3$, we have 
\[
\gamma_r\gamma_s = \sum_{t=0}^3 m_{rst}\gamma_t.
\]
In particular, the matrix $G$ determines an isomorphism between the quaternion algebra over $\QQ$ with multiplication table given by $\{m_{rst}\}$ with either $\End^0(E)$ or its opposite algebra depending on a choice for a square root of $\det G$. 
\end{prop}
\begin{proof}

We have $\gamma_0^2=1$, and for $r\not=0$, we have 
$\Trd(\gamma_r)=0$ so $\gamma_r^2=\Trd(\gamma_r)\gamma_r -\Nrd(\gamma_r)=-\deg(\gamma_r)$.
Therefore 
\[ m_{000}=1,\quad   m_{rrt}= \begin{cases} -\frac{G_{rr}}{2} &: t=0 \text{ and } 1\leq r\leq 3\\ 
0 &: 0\leq r\leq 3 \text{ and } 1\leq t\leq 3.
\end{cases}
\]
Also, note that 
\[
\gamma_r\gamma_s=\widehat{\widehat{\gamma_s}\widehat{\gamma_r}} = \widehat{\gamma_s\gamma_r}=\Trd(\gamma_s\gamma_r)-\gamma_s\gamma_r,\]
so if $\gamma_r\gamma_s=\sum_tm_{rst}\gamma_t$ then 
\[
\gamma_s\gamma_r=\Trd(\gamma_r\gamma_s)-m_{rs0} -\sum_{t=1}^3m_{rst}\gamma_{t}.
\]
Therefore 
\[
m_{srt}=\begin{cases}
\Trd(\gamma_r\gamma_s)-m_{rs0}
&: t=0 \\
-m_{rst} &: 1\leq t \leq 3.
\end{cases}
\]
We conclude that it suffices to calculate $m_{rst}$ for $1\leq r<s\leq 3$. 

By pairing both sides of $\gamma_r\gamma_s=\sum_{t=0}^3m_{rst}\gamma_t$ against $\gamma_k$ for $k=0,1,2,3$, we obtain for each pair $(r,s)$ satisfying $1\leq r<s\leq3$ a system of four equations in the 
indeterminates $m_{rs0}, m_{rs1}, m_{rs2}, m_{rs3}$:
\begin{equation}\label{eqn:multtableeqn}
\Trd(\gamma_r\gamma_s\widehat{\gamma_k}) = \sum_{t=0}^3 m_{rst} \Trd(\gamma_t\widehat{\gamma_k}).
\end{equation}

We will show that  the entries of $G$  determine the left-hand side of Equation~(\ref{eqn:multtableeqn}) uniquely (up to a choice of a square root of $\det(G)$). 
We compute the left hand side of Equation~(\ref{eqn:multtableeqn}) for each $0\leq r<s\leq 3$, $0\leq k\leq 3$. We have that 
\[
\Trd(\gamma_r\gamma_s\widehat{\gamma_r})=\deg(\gamma_r)\Trd(\gamma_s)=\frac{1}{2}\Trd(\gamma_r\widehat{\gamma_r})\Trd(\gamma_s\widehat{1})=\frac{1}{2}G_{rr}G_{s1},
\]
and similarly
$\Trd(\gamma_r\gamma_s\widehat{\gamma_s})=\Trd(\gamma_r)\deg(\gamma_s)=\frac{1}{2}G_{ss}G_{r1}$, for $1\leq r<s\leq 3$. 
Finally, since $\gamma_0=1$,
\[
\Trd(\gamma_r\gamma_s\widehat{\gamma_0})=\Trd(\gamma_r\gamma_s)=-\Trd(\gamma_r\widehat{\gamma_s})=-G_{rs}.
\]
We are left with the case that $\{r,s,k\}$ is a permutation of $\{1,2,3\}$. First, we calculate $\Trd(\gamma_1\gamma_2\widehat{\gamma_3})$. 
For that, we recall the following trilinear form on 
the quaternion algebra $\End^0(E)$: for elements $\alpha_1,\alpha_2,\alpha_3\in \End^0(E)$, define 
\[
m(\alpha_1,\alpha_2,\alpha_3)=\Trd((\alpha_1\alpha_2-\alpha_2\alpha_1)\widehat{\alpha_3}).
\] 
Using the fact that $\widehat{\gamma_i}=-\gamma_i$ and that for elements $\alpha,\beta \in B$ we have $\Trd(\alpha\beta)=\Trd(\beta\alpha)$ and 
$\Trd(\widehat{\alpha})=\Trd(\alpha)$, a calculation shows 
\[
m(\gamma_1,\gamma_2,\gamma_3) = \Trd((\gamma_1\gamma_2-\gamma_2\gamma_1)\widehat{\gamma_3}) = 2\Trd(\gamma_1\gamma_2\widehat{\gamma_3}).
\]The proof of Lemma~15.4.7 in~\cite{Voight} shows that, for any elements $\alpha_0=1,\alpha_1,\alpha_2,\alpha_3$ in a quaternion algebra $B$, we have 
\[
m(\alpha_1,\alpha_2,\alpha_3)^2 =\det((\Trd(\alpha_i\widehat{\alpha_j}))_{0\leq i,j\leq 3}).
\]
We conclude that $m(\gamma_1,\gamma_2,\gamma_3)^2=\det(G)$. We have that $m(\gamma_{\sigma(1)},\gamma_{\sigma(2)},\gamma_{\sigma(3)})=\sgn(\sigma)m(\gamma_1,\gamma_2,\gamma_3)$ for any $\sigma \in S_3$, e.g. by checking this for the three transpositions of $S_3$. The upshot is that we can make a consistent choice of values for 
$\Trd(\gamma_{\sigma(1)}\gamma_{\sigma(2)}\widehat{\gamma_{\sigma(3)}})$ by choosing, for example, $\Trd(\gamma_1\gamma_2\widehat{\gamma_3})=\frac{1}{2}\sqrt{\det(G)}$ and then setting
\[
\Trd(\gamma_{\sigma(1)}\gamma_{\sigma(2)}\widehat{\gamma_{\sigma(3)}})=\frac{\sgn(\sigma)}{2}\sqrt{\det(G)}.
\]
With linear algebra over $\QQ$, we solve the above three systems of four equations to compute all coefficients $m_{rst}$ with $1\leq r<s\leq 3$.  With our earlier calculations, this determines a complete multiplication table which gives $\Lambda\otimes \QQ$ the structure of a quaternion algebra whose underlying quadratic space is isomorphic to $(\QQ^4,G)\cong (\End^0(E),\Nrd)$. 
\end{proof}
\begin{rmk}
    We encounter the same phenomenon we observed in Section~\ref{sec:GS}: we must choose a sign for a square root to determine the multiplication table. The choice of sign of a square root of $\det(G)$ corresponds to the choice of an 
isomorphism of the quaternion algebra $\Lambda\otimes \QQ$ equipped with the multiplication table $\{m_{rst}\}$ with either $\End^0(E)$ or $(\End^0(E))^{\text{op}}$. 
\end{rmk}

\begin{rmk}
We could eliminate this ambiguity by computing $\Trd(\gamma_1\gamma_2\widehat{\gamma_3})$ directly via Schoof's algorithm. 
\end{rmk}
\subsubsection{Gram--Schmidt versus multiplication tables}
One may ask if the approaches in Sections~\ref{sec:GS} and~\ref{sec:multtable} for obtaining a quaternion algebra from the basis $\{\gamma_0,\gamma_1,\gamma_2,\gamma_3\}$ with Gram matrix $G$ are compatible. This is the case: first of all, $G$ determines the 
structure of $\End^0(E)$ as a quadratic space, and by \cite[Proposition 5.2.4]{Voight}, there are only two (up to isomorphism) quaternion algebras with underlying quadratic spaces isomorphic to $(\End^0(E),\Nrd)\cong (\QQ^4,G)$. We can make this explicit, and in fact the choices of square root in each approach are consistent with one another. 

Let $LDL^T=G$ with $D$ diagonal and $L$ lower-triangular with $1$'s on its diagonal. Let $\{\tilde{\gamma_0},\tilde{\gamma_1},\tilde{\gamma_2},\tilde{\gamma_3}\}$ and $a,b,c,d_1,d_2,d_3\in \QQ$ be defined as in Section~\ref{sec:GS}. Then by~\cite[15.4.5]{Voight},
\[
m(\tilde{\gamma_1},\tilde{\gamma_2},\tilde{\gamma_3}) = \det(L)m(\gamma_1,\gamma_2,\gamma_3)=m(\gamma_1,\gamma_2,\gamma_3).
\]
On the other hand, we have $\tilde{\gamma_3}=c\tilde{\gamma_1}\tilde{\gamma_2}$, so 
\[
m(\tilde{\gamma_1},\tilde{\gamma_2},\tilde{\gamma_3}) = 4abc,
\] 
and $4ab>0$, so the sign of $m(\gamma_1,\gamma_2,\gamma_3)$ and the sign of $c$ are equal. The choice of sign for a square root of $\det(G)=m(\gamma_1,\gamma_2,\gamma_3)^2$ is therefore consistent with a choice of 
sign of the square root of 
\[
\frac{2d_3}{d_1d_2} = \frac{\det(G)}{16a^2b^2} = \frac{1}{(4ab)^2}\left(m(\tilde{\gamma_1},\tilde{\gamma_2},\tilde{\gamma_3})\right)^2.
\]

\subsection{Computing an order \texorpdfstring{$\OO$}{O} in \texorpdfstring{$\ZZ+P$}{Z+P}}
We assume that we have computed three inseparable endomorphisms $\gamma_1,\gamma_2,\gamma_3$ such that $\gamma_0\coloneqq1$ and $\gamma_1,\gamma_2,\gamma_3$ generate a lattice $\Lambda$ inside $\End(E)$, along with the Gram matrix $G=(\Trd(\gamma_i\widehat{\gamma_j}))_{0\leq i,j\leq 3}$ and isomorphisms of quadratic spaces $f\colon (\QQ^4,G) \to H(a,b)$ and $g\colon (\End^0(E),\Nrd) \to (\QQ^4,G)$, where $g(\gamma_r)=e_r$, the $r$th standard basis vector of $\QQ^4$. For $0\leq r,s,t\leq 3$, let $m_{rst}\in \QQ$ be  the elements of the multiplication table for the basis $\B=\{\gamma_0,\gamma_1,\gamma_2,\gamma_3\}$: for $0\leq r,s\leq 3$, we have 
\[
\gamma_r\gamma_s = \sum_{t=0}^3 m_{rst}\gamma_t.
\]  Let $M_r$ be the matrix $M_r=(m_{rst})_{0\leq s,t\leq 3}$. From this data, we can compute a basis for $\OO$, the minimal order in $\End(E)$ containing $\Lambda$. The order $\OO$ is generated as a $\ZZ$-module by $\gamma_0,\gamma_1,\gamma_2,$ and $\gamma_3$ and their products. 

We compute a basis for $\OO$ in which basis elements are represented as linear combinations of the $\gamma_i$ as follows. Let $M_{rs}$ denote the $s$th row of $M_r$. Define the $12\times 4$ matrix $A$ to have rows given by the rows of $M_0$, i.e. the $4\times 4$ identity matrix, followed by $M_{rs}$ for $0<r<s\leq 3$. Let $H$ be the  Hermite normal form of $A$. Let $B=(b_{ij})_{0\leq i,j\leq 3}\in M_4(\ZZ)$ be the matrix whose rows are the top four rows of $H$. The rows of $B$ form a lattice $L$ in $\QQ^4$ such that $g^{-1}(L)= \OO$. In particular, if we define  $\beta_i=\sum_{j=0}^3 b_{ij}\gamma_i$ for $0\leq i \leq 3$, then $\{\beta_0=1,\beta_1,\beta_2,\beta_3\}$ is a $\ZZ$-basis for $\OO$. 

\subsection{Computing \texorpdfstring{$\ZZ+P$}{Z+P}}
We now assume that we have computed a suborder $\OO$ of $\ZZ+P$ generated by $\gamma_0=1$ and three inseparable endomorphisms $\gamma_1,\gamma_2,\gamma_3$, where $\OO$ is represented by  four vectors $\{(b_{ij})_{0\leq j \leq 3}\}_{0\leq i \leq 3}$ in $\QQ^4$ such that 
$\beta_i\coloneqq \sum_{j=0}^3b_{ij}\gamma_j$ form a $\ZZ$-basis for $\OO$. We proceed to compute $\ZZ+P$ by iteratively computing an additional inseparable endomorphism $\gamma$ and the order $\OO[\gamma]$, defined to be the smallest order containing both $\OO$ and $\gamma$. It suffices to compute a basis for the $\mathbb{Z}$ lattice spanned by $\beta_0,\ldots,\beta_3,\beta_0\gamma,\ldots,\beta_3\gamma$.   The approach is similar to how we computed an order generated by a lattice basis in the previous subsection. We first compute $(c_0,\ldots,c_3)\in \QQ^4$ such that 
\[
\gamma=\sum_{s=0}^3 c_s\gamma_s
\]
by computing the traces $t_r\coloneqq \Trd(\gamma_r\widehat{\gamma})$ for $0\leq r\leq 3$ and then solving the system of equations
\[
t_r = \sum_{s=0}^3 c_s G_{rs}.
\]
Define $M_{\gamma}\coloneqq \sum_{r=0}^3 c_rM_r$. Then the matrix 
\[
M_{\gamma}' \coloneqq H^{-1}MH
\]
gives the action of left multiplication of $\gamma$ on the basis elements $b_r=\sum_{s=0} H_{rs}\gamma_s$ for $\OO$. Let $A$ be the matrix whose rows are the rows of $H$ and the rows of $M_{\gamma'}$. The top four rows of the Hermite normal form of $A$ yield a basis for a lattice $L_{\OO[\gamma]}$ in $\QQ^4$ such that $g^{-1}(L_{\OO[\gamma]})=\OO$. We then define $B$ to be the top four rows of $H$. 

We remark that we can check whether the order $\OO$ given by the matrix $B$ is equal to $\ZZ+P$ by simply computing its discriminant and checking if the discriminant is $p^4$. The discriminant of $\OO$ is the determinant of the Gram matrix $B^TGB$. 

\subsection{From \texorpdfstring{$\ZZ+P$}{Z+P} to \texorpdfstring{$\End(E)$}{End(E)}}\label{sec:voight-algorithm}
Assume that we have computed three inseparable endomorphisms $\gamma_i$ for $1\leq i \leq 3$, the Gram matrix $G=(\Trd(\gamma_i\widehat{\gamma_j}))$, and a matrix $B=(b_{ij})\in M_4(\QQ)$ so that $\beta_i\coloneqq \sum_{j=0}^3 b_{ij}\gamma_i$ form the $\ZZ$-basis for $\ZZ+P$. Then we have so far computed a basis for the unique order of index $p$ in $\End(E)$, according to Proposition~\ref{prop:z+p}. Using results and algorithms in~\cite{Voight}, we only need a little linear algebra to efficiently compute a basis for $\End(E)$. We recall the notion of a $p$-saturated order from~\cite{Voight} below, and show that in our case, a $p$-saturated order containing $\ZZ+P$ is $\End(E)$. 
\begin{defn}
Let $p$ be an odd prime. An order $\mathcal O\subseteq B$ is said to be $p$\emph{-saturated}  if $\mathcal O_p\coloneqq\mathcal O\otimes\mathbb Z_p$ has a basis $x_1,x_2,x_3,x_4$ such that the quadratic form $\Nrd\colon \OO_p \to \QQ_p$ is diagonal with respect to that basis and such that $v_{p}(\Nrd(x_i))\leq 1 $ for all $1\leq i\leq 4$.
An order $\mathcal O\subseteq B$ is said to be $p$\emph{-maximal} for a prime $p$ if $\mathcal O_p\coloneqq\mathcal O\otimes\mathbb Z_p$ is maximal in $B\otimes\mathbb Q_p$. 
\end{defn}

The following proposition shows that for quaternion algebras over $\QQ$ ramified at $p$, orders that are $p$-saturated must also be $p$-maximal.

\begin{prop}\label{p-saturated}
Let $B$ be a quaternion algebra over $\QQ$ ramified at $p$. If $\mathcal O\subseteq B$ is a $\ZZ$-order which is $p$-saturated, then $\mathcal O$ is $p$-maximal.
\end{prop}

\begin{proof}
Let $x_0=1,x_1,x_2,x_3$ be a normalized basis of $\OO_p\coloneqq\OO\otimes\mathbb Z_p$ with respect to the quadratic form $\Nrd$ such that $e_i\coloneqq v_{p}(\Nrd(x_i))\leq 1 $ for $i=1,2,3$ and $e_1\leq e_2\leq e_3$.

Then \begin{align*}
\disc(\OO_p) &=
 \det(\Trd(x_i\overline{x_j}))\ZZ_p \\
 &=\det \begin{pmatrix}
    2 & 0 & 0 & 0 \\
    0 & u_1p^{e_1} & 0 & 0\\ 
    0 & 0 & u_2p^{e_2} & 0\\ 
    0 & 0 & 0 & u_3p^{e_3}\\ 
  \end{pmatrix} \ZZ_p \\ 
  &=  p^{e_1+e_2+e_3} \mathbb Z_p\supseteq p^3\ZZ_p, 
  \end{align*}
  where $u_1,u_2,u_3\in \mathbb Z_p^{\times}.$
The discriminant of $\OO_p$ is the square of an ideal in $\mathbb Z_p$, so $e_1+e_2+e_3$ has to be even and therefore is either $0$ or $2$. The first case is not possible since $B$ is ramified at $p$. This implies that $v_{p}(\discrd(\OO_p))=1=v_p(\disc(B))$, so we conclude $\mathcal O \subseteq B$ is $p$-maximal.
\end{proof}

\begin{cor}\label{cor-saturated}
Let $\mathcal O\subseteq \End^0(E)$ be a $p$-saturated order such that $\ZZ+P \subseteq\mathcal O$. Then $\mathcal O=\End(E).$
\end{cor}

\begin{proof}
By Proposition~\ref{prop:z+p}, the order $\mathbb Z+P$  is locally maximal at all primes $\ell\neq p$. Since $\ZZ+P \subseteq\mathcal O\subseteq \End(E)$, the order $\mathcal O$ is also $\ell$-maximal for all $\ell\neq p$.
Moreover $\mathcal O$ is $p$-saturated, so $\mathcal O$  is $p$-maximal by Proposition~\ref{p-saturated}. This implies that $\mathcal O$ is maximal in $\End^0(E)$ and, since $\ZZ+P \subseteq\mathcal O$, by Proposition $\ref{prop:z+p}$ we have $\mathcal O=\End(E)$.
\end{proof}

Therefore given the order $\mathbb Z+P$, we can recover  the maximal order containing $\mathbb Z+P $  by  computing the $p$-saturated order that contains $\mathbb Z+P$. This is done efficiently with Algorithm 3.12 and Algorithm 7.9 in~\cite{Voight}.

\newcommand{\etalchar}[1]{$^{#1}$}


\begin{thebibliography}{BCNE{\etalchar{+}}19}

\bibitem[ACNL{\etalchar{+}}23]{ACLLNSS}
Sarah Arpin, Catalina Camacho-Navarro, Kristin Lauter, Joelle Lim, Kristina
  Nelson, Travis Scholl, and Jana Sotáková.
\newblock Adventures in supersingularland.
\newblock {\em Experimental Mathematics}, 32(2):241--268, 2023.

\bibitem[BCC{\etalchar{+}}23]{secuer}
Andrea Basso, Giulio Codogni, Deirdre Connolly, Luca De~Feo, Tako~Boris
  Fouotsa, Guido~Maria Lido, Travis Morrison, Lorenz Panny, Sikhar Patranabis,
  and Benjamin Wesolowski.
\newblock Supersingular curves you can trust.
\newblock In {\em Advances in Cryptology – EUROCRYPT 2023: 42nd Annual
  International Conference on the Theory and Applications of Cryptographic
  Techniques, Lyon, France, April 23–27, 2023, Proceedings, Part II}, page
  405–437, Berlin, Heidelberg, 2023. Springer-Verlag.

\bibitem[BCNE{\etalchar{+}}19]{BCEMP}
Efrat Bank, Catalina Camacho-Navarro, Kirsten Eisentr{\"a}ger, Travis Morrison,
  and Jennifer Park.
\newblock Cycles in the supersingular $\ell$-isogeny graph and corresponding
  endomorphisms.
\newblock In Jennifer~S. Balakrishnan, Amanda Folsom, Matilde Lal{\'i}n, and
  Michelle Manes, editors, {\em Research Directions in Number Theory}, pages
  41--66, Cham, 2019. Springer International Publishing.

\bibitem[BLS12]{BLS}
Reinier Br\"{o}ker, Kristin Lauter, and Andrew~V. Sutherland.
\newblock Modular polynomials via isogeny volcanoes.
\newblock {\em Math. Comp.}, 81(278):1201--1231, 2012.

\bibitem[Br{\"o}09]{Brok09}
Reinier Br{\"o}ker.
\newblock Constructing supersingular elliptic curves.
\newblock {\em J. Comb. Number Theory}, 1(3):269--273, 2009.

\bibitem[Brz90]{Brz90}
J.~Brzezinski.
\newblock On automorphisms of quaternion orders.
\newblock {\em J. Reine Angew. Math.}, 403:166--186, 1990.

\bibitem[BS11]{BS2011}
Gaetan Bisson and Andrew~V. Sutherland.
\newblock Computing the endomorphism ring of an ordinary elliptic curve over a
  finite field.
\newblock {\em J. Number Theory}, 131(5):815--831, 2011.

\bibitem[CGL09]{CGL2009}
Denis~X. Charles, Eyal~Z. Goren, and Kristin Lauter.
\newblock Cryptographic hash functions from expander graphs.
\newblock {\em J. Cryptology}, 22(1):93--113, 2009.

\bibitem[CKMZ22]{CKMZ22}
T\'{\i}mea Csah\'{o}k, P\'{e}ter Kutas, Micka\"{e}l Montessinos, and Gergely
  Z\'{a}br\'{a}di.
\newblock Explicit isomorphisms of quaternion algebras over quadratic global
  fields.
\newblock {\em Res. Number Theory}, 8(4):Paper No. 77, 24, 2022.

\bibitem[CS21]{CS21}
Mathilde Chenu and Benjamin Smith.
\newblock {Higher-degree supersingular group actions}.
\newblock {\em {Mathematical Cryptology}}, 2021.

\bibitem[CSCS22]{CCS22}
Maria Corte{-}Real~Santos, Craig Costello, and Jia Shi.
\newblock Accelerating the delfs-galbraith algorithm with fast subfield root
  detection.
\newblock In Yevgeniy Dodis and Thomas Shrimpton, editors, {\em Advances in
  Cryptology - {CRYPTO} 2022 - 42nd Annual International Cryptology Conference,
  {CRYPTO} 2022, Santa Barbara, CA, USA, August 15-18, 2022, Proceedings, Part
  {III}}, volume 13509 of {\em Lecture Notes in Computer Science}, pages
  285--314. Springer, 2022.

\bibitem[CSV21]{CSV}
Sara Chari, Daniel Smertnig, and John Voight.
\newblock On basic and {B}ass quaternion orders.
\newblock {\em Proc. Amer. Math. Soc. Ser. B}, 8:11--26, 2021.

\bibitem[DFKL{\etalchar{+}}20]{sqisign}
Luca De~Feo, David Kohel, Antonin Leroux, Christophe Petit, and Benjamin
  Wesolowski.
\newblock {SQISign}: compact post-quantum signatures from quaternions and
  isogenies.
\newblock In {\em Advances in cryptology -- ASIACRYPT 2020. 26th international
  conference on the theory and application of cryptology and information
  security, Daejeon, South Korea, December 7--11, 2020. Proceedings. Part I},
  pages 64--93. Cham: Springer, 2020.

\bibitem[DG16]{DG16}
Christina Delfs and Steven~D. Galbraith.
\newblock Computing isogenies between supersingular elliptic curves over
  $\mathbb{F}_p$.
\newblock {\em Des. Codes Cryptography}, 78(2):425--440, February 2016.

\bibitem[EHL{\etalchar{+}}18]{EHLMP}
Kirsten Eisentr{\"a}ger, Sean Hallgren, Kristin Lauter, Travis Morrison, and
  Christophe Petit.
\newblock Supersingular isogeny graphs and endomorphism rings: Reductions and
  solutions.
\newblock In Jesper~Buus Nielsen and Vincent Rijmen, editors, {\em Advances in
  Cryptology -- EUROCRYPT 2018}, pages 329--368, Cham, 2018. Springer
  International Publishing.

\bibitem[EHL{\etalchar{+}}20]{EHLMP20}
Kirsten Eisentr\"{a}ger, Sean Hallgren, Chris Leonardi, Travis Morrison, and
  Jennifer Park.
\newblock Computing endomorphism rings of supersingular elliptic curves and
  connections to path-finding in isogeny graphs.
\newblock In {\em A{NTS} {XIV}---{P}roceedings of the {F}ourteenth
  {A}lgorithmic {N}umber {T}heory {S}ymposium}, volume~4 of {\em Open Book
  Ser.}, pages 215--232. Math. Sci. Publ., Berkeley, CA, 2020.

\bibitem[Eic36]{Eic36}
Martin Eichler.
\newblock Untersuchungen in der {Z}ahlentheorie der rationalen
  {Q}uaternionenalgebren.
\newblock {\em J. Reine Angew. Math.}, 174:129--159, 1936.

\bibitem[Elk98]{Elk95}
Noam~D. Elkies.
\newblock Elliptic and modular curves over finite fields and related
  computational issues.
\newblock In {\em Computational perspectives on number theory ({C}hicago, {IL},
  1995)}, volume~7 of {\em AMS/IP Stud. Adv. Math.}, pages 21--76. Amer. Math.
  Soc., Providence, RI, 1998.

\bibitem[EPSV23]{EPSV}
Jonathan~Komada Eriksen, Lorenz Panny, Jana Sotáková, and Mattia Veroni.
\newblock Deuring for the people: Supersingular elliptic curves with prescribed
  endomorphism ring in general characteristic.
\newblock Cryptology ePrint Archive, Paper 2023/106, 2023.
\newblock \url{https://eprint.iacr.org/2023/106}.

\bibitem[ES24]{ES24}
Kirsten Eisentr{\"a}ger and Gabrielle Scullard.
\newblock Connecting {K}ani's lemma and path-finding in the bruhat-tits tree to
  compute supersingular endomorphism rings, 2024.
\newblock \url{https://arxiv.org/abs/2402.05059}.

\bibitem[Gal12]{GalPKC}
Steven~D. Galbraith.
\newblock {\em Mathematics of public key cryptography}.
\newblock Cambridge University Press, Cambridge, 2012.

\bibitem[GPS17]{GPS}
Steven~D. Galbraith, Christophe Petit, and Javier Silva.
\newblock Identification protocols and signature schemes based on supersingular
  isogeny problems.
\newblock In {\em Advances in cryptology---{ASIACRYPT} 2017. {P}art {I}},
  volume 10624 of {\em Lecture Notes in Comput. Sci.}, pages 3--33. Springer,
  2017.

\bibitem[HvdH21]{HarveyH21}
David Harvey and Joris van~der Hoeven.
\newblock Integer multiplication in time {$O(n \log n)$}.
\newblock {\em Ann. of Math. (2)}, 193(2):563--617, 2021.

\bibitem[HW08]{HW08}
Godfrey~H. Hardy and Edward~M. Wright.
\newblock {\em An introduction to the theory of numbers}.
\newblock Oxford University Press, Oxford, sixth edition, 2008.

\bibitem[KLPT14]{KLPT}
David Kohel, Kristin Lauter, Christophe Petit, and Jean-Pierre Tignol.
\newblock On the quaternion {$\ell$}-isogeny path problem.
\newblock {\em LMS Journal of Computation and Mathematics}, 17:418--432, 2014.

\bibitem[Koh96]{Koh96}
David Kohel.
\newblock {\em Endomorphism rings of elliptic curves over finite fields}.
\newblock PhD thesis, University of California, Berkeley, 1996.

\bibitem[Lit28]{Lit28}
John~E. Littlewood.
\newblock On the class-number of the corpus {$P(\sqrt{-k})$}.
\newblock {\em Proc. London Math. Soc. (2)}, 27(5):358--372, 1928.

\bibitem[LP92]{LP92}
H.~W. Lenstra, Jr. and Carl Pomerance.
\newblock A rigorous time bound for factoring integers.
\newblock {\em J. Amer. Math. Soc.}, 5(3):483--516, 1992.

\bibitem[Mes86]{Mes86}
J.-F. Mestre.
\newblock La m\'ethode des graphes. {E}xemples et applications.
\newblock In {\em Proceedings of the international conference on class numbers
  and fundamental units of algebraic number fields ({K}atata, 1986)}, pages
  217--242. Nagoya Univ., Nagoya, 1986.

\bibitem[MSS16]{morain:hal-01320388}
Fran{\c c}ois Morain, Charlotte Scribot, and Benjamin Smith.
\newblock {Computing cardinalities of Q-curve reductions over finite fields}.
\newblock {\em {LMS Journal of Computation and Mathematics}}, 19(A):15, August
  2016.

\bibitem[Piz80a]{Piz80}
Arnold Pizer.
\newblock An algorithm for computing modular forms on {$\Gamma _{0}(N)$}.
\newblock {\em J. Algebra}, 64(2):340--390, 1980.

\bibitem[Piz80b]{Piz80theta}
Arnold Pizer.
\newblock Theta series and modular forms of level {$p^{2}M$}.
\newblock {\em Compositio Math.}, 40(2):177--241, 1980.

\bibitem[PW23]{PW24}
Aurel Page and Benjamin Wesolowski.
\newblock The supersingular endomorphism ring and one endomorphism problems are
  equivalent.
\newblock Cryptology ePrint Archive, Paper 2023/1399, 2023.
\newblock \url{https://eprint.iacr.org/2023/1399}.

\bibitem[Rab80]{Rabin80}
Michael~O. Rabin.
\newblock Probabilistic algorithms in finite fields.
\newblock {\em SIAM J. Comput.}, 9(2):273--280, 1980.

\bibitem[Rob22]{Rob23applications}
Damien Robert.
\newblock Some applications of higher dimensional isogenies to elliptic curves
  (overview of results).
\newblock Cryptology ePrint Archive, Paper 2022/1704, 2022.
\newblock \url{https://eprint.iacr.org/2022/1704}.

\bibitem[Sch95]{SchoofSEA}
Ren\'{e} Schoof.
\newblock Counting points on elliptic curves over finite fields.
\newblock {\em J. Th\'{e}or. Nombres Bordeaux}, 7(1):219--254, 1995.
\newblock Les Dix-huiti\`emes Journ\'{e}es Arithm\'{e}tiques (Bordeaux, 1993).

\bibitem[Sil09]{AEC}
Joseph~H. Silverman.
\newblock {\em The arithmetic of elliptic curves}.
\newblock Springer, New York, 2009.

\bibitem[{The}22]{sagemath}
{The Sage Developers}.
\newblock {\em {S}ageMath, the {S}age {M}athematics {S}oftware {S}ystem
  ({V}ersion 9.7)}, 2022.
\newblock {\tt https://www.sagemath.org}.

\bibitem[Voi13]{V2013}
John Voight.
\newblock Identifying the matrix ring: algorithms for quaternion algebras and
  quadratic forms.
\newblock {\em Developments in Mathematics}, 31:255--298, 2013.

\bibitem[Voi21]{Voight}
John Voight.
\newblock {\em Quaternion algebras}, volume 288 of {\em Graduate Texts in
  Mathematics}.
\newblock Springer, Cham, 2021.

\bibitem[vzGG13]{vzGG}
Joachim von~zur Gathen and J\"{u}rgen Gerhard.
\newblock {\em Modern computer algebra}.
\newblock Cambridge University Press, Cambridge, third edition, 2013.

\bibitem[Wes22]{Wes22}
Benjamin Wesolowski.
\newblock {The supersingular isogeny path and endomorphism ring problems are
  equivalent}.
\newblock In {\em {FOCS 2021 - 62nd Annual IEEE Symposium on Foundations of
  Computer Science}}, Denver, Colorado, United States, February 2022.

\end{thebibliography}
\end{document}